\documentclass{article}
\usepackage{amsmath,amssymb,amsthm,mathtools}
\usepackage{hyperref}

\author{Henrik Schneider\\[0.4em]
\small University of Duisburg-Essen, Germany\\
\small\href{mailto:henrik.schneider@uni-due.de}{\texttt{henrik.schneider@uni-due.de}}}
\title{Weak Solutions for the Unregularised Hibler Momentum Equation with Degenerate Coefficients}

\newtheorem{theorem}{Theorem}[section]
\newtheorem{lemma}[theorem]{Lemma}
\theoremstyle{definition}
\newtheorem{definition}{Definition}[section]
\newtheorem{assumption}[definition]{Assumption}

\newtheorem{remark}[definition]{Remark}
\newcommand{\bfu}{\boldsymbol{u}}
\newcommand{\bfv}{\boldsymbol{v}}
\newcommand{\bfw}{\boldsymbol{w}}
\newcommand{\bfsigma}{\boldsymbol{\sigma}}
\newcommand{\bftau}{\boldsymbol{\tau}}
\newcommand{\bfxi}{\boldsymbol{\xi}}
\newcommand{\bfe}{\boldsymbol{e}}
\newcommand{\bfI}{\boldsymbol{\mathrm{I}}}
\newcommand{\bff}{\boldsymbol{f}}
\newcommand{\K}{\mathcal{K}}
\newcommand{\D}{\mathcal{D}}
\newcommand{\dd}{\:\mathrm{d}}
\newcommand{\B}{\mathcal{B}}
\newcommand{\wstar}{\overset{\ast}{\rightharpoonup}}
\newcommand{\bfeta}{\boldsymbol{\eta}}
\newcommand{\bfz}{\boldsymbol{z}}
\newcommand{\Hh}{\mathcal{H}}

\renewcommand\div{\operatorname{\mathrm{div}}}
\DeclareMathOperator{\tr}{tr}
\DeclareMathOperator{\dev}{dev}
\DeclareMathOperator{\dom}{dom}
\DeclareMathOperator{\dist}{dist}
\DeclareMathOperator{\supp}{supp}
\DeclareMathOperator{\esssup}{ess sup}

\begin{document}
\maketitle
\begin{abstract}
We study the momentum equation of the decoupled unregularised Hibler sea-ice
model with non-negative ice mass and ice strength, both of which may vanish.
The ocean drag provides coercivity, while the stress law and mixed boundary
conditions are encoded by a convex dissipation functional. We prove existence,
uniqueness and stability for the time-discrete problem and reconstruct an
admissible stress. A Rothe approximation yields global weak variational
solutions for time-dependent coefficients under a one-sided growth condition
on the mass. Weak solutions are unique when the ice strength is
independent of time.
For spatially Lipschitz ice strength, finite dissipation also yields a local
measure structure of the weighted deformation and a global bound on the
negative part of the weighted divergence.
\end{abstract}
\section{Introduction}
Hibler's viscous-plastic model is a standard description of large-scale sea-ice
dynamics \cite{H79}. The original formulation combines plastic behaviour at
larger deformation rates with viscous creep at small rates by imposing an
upper bound on the viscosities. This regularisation avoids the singularity
of the plastic constitutive law at zero deformation rate.

Hunke and Dukowicz \cite{HD97} introduced the elastic-viscous-plastic (EVP)
model to facilitate explicit time stepping and parallel computation. An
artificial elastic component replaces the instantaneous stress relation by
an evolution equation for the stress. Subsequent developments include
Hunke's revised treatment of the linearisation \cite{H01} and the modified
EVP iteration of Bouillon et al.\ \cite{BFLM13}, designed to improve
convergence towards the viscous-plastic solution. Recent mathematical work
also treats additional Voigt regularisations in the stress equation
\cite{BLTT26a} and in the momentum balance \cite{BLTT26b}.

The unregularised stress law considered here is singular at zero deformation
rate, and the possible disappearance of ice introduces further degeneracy
through the mass and ice strength. These features motivate a variational
formulation that does not require a positive lower bound for either
coefficient.

Strong well-posedness for regularised or modified coupled Hibler models was
established in \cite{BDHH22,LTT22}. The parabolic-hyperbolic setting is treated
in \cite{B25}. More recently, Dingel and Disser \cite{DD25} proved global weak
existence and uniqueness for the momentum equation with cut-offs at small
and large strain rates. Denk, Gmeineder and Hieber \cite{DGH25} treated the
genuinely unregularised stress through energy-driven solutions and relaxed
Hibler energies, including boundary approximation.

In this work the mass and ice strength are prescribed independently of the
velocity. Compared with the setting in \cite{DGH25}, our contribution is an
existence theory for possibly vanishing, time-dependent coefficients with the
full quadratic ocean drag. Its coercivity controls the velocity even where
the mass or ice strength vanishes. We impose mixed Dirichlet and Neumann
conditions, allowing an empty Dirichlet portion.

We first formulate the time-discrete problem as a variational inequality in
$L^3$. A dual support-function representation yields an admissible saturated
stress and encodes the Dirichlet condition through a relaxed boundary term.
For spatially Lipschitz ice strength, we also identify the local measure
structure of the weighted deformation. For the evolution problem, a one-sided
growth condition on the mass permits weighted energy estimates and the
construction of weak variational solutions by time discretisation.

When the ice strength is independent of time, we prove uniqueness in this
weak solution class without requiring a time derivative of the velocity.
The proof combines approximation with convergence of the dissipation and
one-sided time mollification, techniques related to the scalar linear-growth
theory in \cite{E25}. The weak evolution is formulated in terms of the
velocity and the dissipation. A full stress-and-flow-rule formulation for
these solutions, uniqueness for general time-dependent ice strength and the
coupling to the transport equations remain open within this framework.

\section{Model}
The motion of sea ice is modelled by
\begin{align*}
    m \partial_t \bfu  + m f_c \bfe_3 \times \bfu - \bftau_o(\bfu) = \div \bfsigma + \bftau_a - mg\nabla H
\end{align*}
on $\Omega \subset \mathbb R^{2}$, where $m=\rho_{i}h \geq 0$ is the ice mass per unit area, $\bfu$
the ice velocity, $\bfsigma \in \mathbb S$ (symmetric $2\times 2$ matrices) the vertically integrated stress tensor,
$\bftau_a$ the atmospheric drag,
\begin{align*}
    \bftau_o (\bfu) = c_o |\bfu_o-\bfu|(\bfu_o-\bfu), \quad c_o := \rho_o C_o >0
\end{align*}
the ocean drag ($\bfu_o$ the ocean velocity), $mf_c \bfe_3 \times \bfu$ the Coriolis term and $mg\nabla H$ the tilt of the ocean surface.
We neglect the convective term $(\bfu \cdot \nabla)\bfu$ due to scaling properties.
The rheology is viscous-plastic with an elliptic yield curve \cite{H79}.
Let $\varepsilon(\bfu) := \frac{1}{2} (\nabla\bfu+\nabla\bfu^T)$ be the strain rate tensor, $\zeta = P/(2\Delta(\varepsilon(\bfu)))$, $\eta=\zeta/e^2$ ($e=2$) the nonlinear viscosities and
\begin{align*}
    \Delta (\bfxi) = \sqrt{\xi^2_I+ e^{-2}\xi_{II}^2}, \quad
    \xi_I := \tr \bfxi, \quad \xi_{II} := \sqrt{(\xi_{11}-\xi_{22})^2 +4 \xi_{12}^2} \ ,
\end{align*}
then in the plastic regime
\begin{align}
    \bfsigma = 2\eta \varepsilon(\bfu) + (\zeta-\eta) \tr \varepsilon(\bfu) \bfI - \frac{P}{2}\bfI, \quad
    P= P^*h \exp(-C(1-A)) \label{eq:plastic_stress}
\end{align}
where $P\geq 0$ is the compressive sea ice strength ($h$ sea ice thickness, $A$ sea ice concentration). The balance laws for $A$ and $h$
are assumed to be decoupled, i.e.\ $m$ and $P$ are given data.
The stress invariants are
\begin{align*}
    \sigma_I := \frac{1}{2} \tr \bfsigma , \qquad \sigma_{II} := \sqrt{\frac{1}{4}(\sigma_{11}-\sigma_{22})^2 +\sigma_{12}^2}
\end{align*}
For $P>0$ and nonzero deformation rate, \eqref{eq:plastic_stress}
places $\bfsigma$ on the boundary of the yield set
\begin{align}
    F(\bfsigma, P) := \frac{(\sigma_I + P/2)^2}{(P/2)^2} + \frac{\sigma_{II}^2}{(P/(2e))^2} -1 \leq 0 \label{eq:def_F}
\end{align}
with the deformation rate following the normal flow rule. For $P>0$, the
regularisation $\Delta\mapsto\max(\Delta,\Delta_{\mathrm{min}})$ places
$\bfsigma$ strictly inside the ellipse when $\Delta<\Delta_{\mathrm{min}}$
and on its boundary when $\Delta\geq\Delta_{\mathrm{min}}$. The unregularised
model considered here is described by the corresponding complementarity formulation. With
the admissible convex set
\begin{align*}
    &\K(P) := \{\bftau \in \mathbb S \ : \ F(\bftau, P) \leq 0\} \quad \text{if} \ P>0, \\
    &\K(0) := \{0\}
\end{align*}
the flow rule reads
\begin{align*}
    \varepsilon(\bfu)(x) \in \mathcal{N}_{\K(P(x))}(\bfsigma(x))
\end{align*}
where $\mathcal{N}_{\K}$ denotes the normal cone to $\K$.
\subsection{Time discretisation}
We apply an implicit Euler discretisation with time step $k>0$.
Given $\bfu^{n-1}$, $m$, $P$, $\bfu_o$ and $\bff := \bftau_a -mg\nabla H$, find
$(\bfsigma, \bfu)$ such that
\begin{subequations}
\label{eq:time_discrete_model}
\begin{align}
    \frac{m}{k} (\bfu-\bfu^{n-1}) + m f_c \bfe_3 \times \bfu - \div \bfsigma - \bftau_o(\bfu) = \bff \quad \text{in} \ \Omega, \\
    \varepsilon(\bfu) \in \mathcal{N}_{\K(P)}(\bfsigma)\\
    \bfu = 0 \quad \text{on} \ \Gamma_D \\
    \bfsigma \nu = 0 \quad \text{on} \ \Gamma_N
\end{align}
\end{subequations}

where $\partial\Omega = \bar\Gamma_D \cup \bar\Gamma_N$.
The Dirichlet boundary condition is imposed in the relaxed form given in \eqref{eq:int_darstellung}.
If the Coriolis term is treated explicitly ($\bfu^{n-1}$ instead of $\bfu$),
then the corresponding relaxed variational problem is a convex minimisation
problem. The fully implicit case is a
non-variational but monotone problem. Its well-posedness is covered in Section \ref{sec:td}.
\section{Assumptions}
For $\bfxi \in \mathbb S$ we have the decomposition $\bfxi = \frac{\xi_I}{2}\bfI + \dev \bfxi$, where $\xi_I = \tr \bfxi$ is the trace and $\dev\bfxi = \bfxi- \frac{1}{2}\tr \bfxi \bfI$. We have for the invariant $\xi_{II}$ that $|\dev \bfxi |_F^2 = \frac{1}{2}\xi_{II}^2$
with $|\cdot|_F$ the Frobenius norm. For the stress invariants we have
\begin{align*}
    |\dev\bftau|_F^2 = 2 \tau_{II}^2  \\
    \bftau : \bfxi = \tau_I \xi_I + \dev\bftau : \dev\bfxi, \quad
    \dev\bftau : \dev \bfxi \leq |\dev \bftau|_F |\dev \bfxi|_F = \tau_{II} \xi_{II}
\end{align*}
with $\bftau = \tau_I \bfI + \dev\bftau$. The function $\bfxi \mapsto\Delta(\bfxi)$ defines a norm $|\cdot|_\Delta$ on $\mathbb S$, the corresponding
dual norm is
\begin{align*}
    |\bftau|_{\Delta^*} = \sqrt{\tau_I^2 + e^2\tau_{II}^2} \ .
\end{align*}
For all $\bfxi\in\mathbb S$ we also have
\begin{align}
    |\xi_I| \leq \Delta(\bfxi) \label{eq:xi_estimate}
\end{align}
with equality exactly when $\dev \bfxi = 0$.

\begin{lemma} \label{lema:prop_K}
    For $P\geq 0$ the following statements hold:
    \begin{enumerate}
        \item $\K(P) = P \K(1)$, in particular $P\mapsto \K(P)$ is monotone, i.e.\ from $0\leq P' \leq P$ follows $\K(P') \subseteq \K(P)$, and $0\in \K(P)$ for all $P\geq 0$.
        \item Boundedness: $|\bftau |_F \leq c_\K P$ for all $\bftau \in \K(P)$ with $c_\K = \sqrt{2+1/(2e^2)}$.
        \item For $P>0$, $\K(P)$, understood in the three-dimensional space $(\tau_I, \dev\bftau) \in \mathbb R \times \mathbb R^2$, is an ellipsoid with centre $-\frac{P}{2} \bfI$, hence strictly convex with $C^\infty$ boundary, and the outer unit normal $\partial \K(P) \rightarrow S^2$ is bijective.
    \end{enumerate}
\end{lemma}

\begin{proof}
    For \textit{1.}, the case $P=0$ follows directly from $\K(0)=\{0\}$.
    For $P>0$ we use the scaling of the invariants,
    \begin{align*}
        (P\bftau)_I = P\tau_I \quad (P\bftau)_{II} = P\tau_{II},
    \end{align*}
    which leads directly to $F(P\bftau, P) = F(\bftau, 1)$ and consequently $P\bftau \in \K(P) \Leftrightarrow \bftau \in \K(1)$. The monotonicity
    follows from $\K(P') = \frac{P'}{P}\K(P) \subseteq \K(P)$, since $\K(P)$ is convex and $0\in\K(P)$. Finally $F(0, P) = 0$ for all $P>0$, so $0$ lies on the boundary of $\K(P)$, and $0\in \K(0)$ by definition.

    The definition of $F$ in \eqref{eq:def_F} shows $\tau_I \in [-P,0]$ and $|\tau_{II}|\leq P/(2e)$ for $\bftau \in \K(P)$. Then the boundedness follows from the estimate
    \begin{align*}
        |\bftau|_F^2 = 2|\tau_I|^2 + 2|\tau_{II}|^2 \leq 2 P^2 + \frac{P^2}{2e^2} \ .
    \end{align*}
    The last claim follows from the fact that $F(\cdot, P)$ is a quadratic polynomial with positive definite principal part in the coordinates $(\tau_I, \dev\bftau)$. Then the set $\{F\leq 0\}$ is a non-degenerate ellipsoid and the Gauss map is a diffeomorphism onto $S^2$.
\end{proof}
We impose the following assumptions for the time-discrete problem.
\begin{assumption}\label{ass}
    \begin{enumerate}
        \item[(A1)] Let $\Omega \subset \mathbb R^2$ be a bounded Lipschitz domain and $\partial \Omega = \bar \Gamma_D \cup \bar \Gamma_N$ with relatively open and disjoint $\Gamma_D, \Gamma_N$ with $\mathcal{H}^1(\partial \Omega \setminus (\Gamma_D \cup \Gamma_N))=0$. The case $\Gamma_D = \emptyset$ is allowed.
        \item[(A2)] Let $m \in L^\infty(\Omega)$, $m\geq 0$, $c_o>0$ constant and $f_c \in L^\infty(\Omega)$.
        \item[(A3)] Let $P \in C(\bar \Omega)$ and $P\geq0$, in particular $P=0$
        on subsets (ice boundary and open water) is allowed.
        \item[(A4)] Let $\sqrt{m} \bfu^{n-1} \in L^2(\Omega)^2$, $\bfu_o \in L^3(\Omega)^2$ and $\bff \in L^{3/2}(\Omega)^2$.
    \end{enumerate}
\end{assumption}
Our solution space for the velocity,
\begin{align*}
    V := L^3(\Omega)^2, \quad V^* = L^{3/2}(\Omega)^2 \ ,
\end{align*}
is reflexive and separable. We use the same norm notation for scalar,
vector-valued and tensor-valued spaces, and omit the domain when it is clear
from the context. The admissible stress set is defined by
\begin{align}
    \begin{split}
    \Sigma := \{ \bftau \in L^\infty(\Omega; \mathbb S) : {}&\bftau(x) \in \K(P(x)) \text{ a.e.},\\
    &\div \bftau \in L^{3/2}(\Omega)^2 \ \text{and}\\
    &\bftau\nu_{|\Gamma_N} = 0 \ \text{weakly} \} \ .
    \end{split}\label{eq:def_Sigma}
\end{align}

The ocean drag $\bftau_o$ alone provides coercivity, so neither Poincar\'e's nor Korn's inequality is required. No positive lower bound for $m$ or $P$ is needed. By the pointwise condition in the definition of $\Sigma$ \eqref{eq:def_Sigma} together with Lemma \ref{lema:prop_K} we have $|\bftau|_F\leq c_\K P$. In particular, $\bftau=0$ a.e.\ on $\{P=0\}$. A traction condition on an
internal interface requires an appropriate normal trace and is not inferred
from this pointwise statement for an arbitrary zero set.

\section{Dissipation function}
The support function of $\K(P)$ can be computed explicitly. For the convex
duality underlying this representation, see \cite[Chapter I]{ET99}.
\begin{lemma}\label{lem:prop_D}
    For $P\geq 0$ and $\bfxi \in \mathbb S$ we have
    \begin{align}
        D(P,\bfxi) := \sup_{\bftau\in\K(P)} \bftau : \bfxi = \frac{P}{2}(\Delta(\bfxi)-\xi_I) \ .
        \label{eq:def_D}
    \end{align}
    In particular $D(P,\cdot)$ is convex, homogeneous of degree 1, non-negative,
    Lipschitz continuous with constant $c_\K P$, and we have
    \begin{align}
        D(P,\bfxi) = 0 \Leftrightarrow \bfxi = \alpha \bfI \ \text{for some}\ \alpha \geq 0 \label{eq:dissipationfree}
    \end{align}
    for all $P>0$. Thus isotropic expansion is dissipation-free. Positive
divergence alone does not imply zero dissipation. This property does not
exclude positive principal stresses elsewhere on the yield ellipse.
Additionally the pointwise estimate
    \begin{align}
        D(P, \bfxi) \geq P(\xi_I)^- \label{eq:negativ_part_D}
    \end{align}
    holds, where $(\cdot)^-$ is the negative part.
\end{lemma}
\begin{proof}
    First we show \eqref{eq:def_D}. For $P=0$ it follows immediately. For $P>0$ we estimate with Cauchy--Schwarz and
    the ellipse condition \eqref{eq:def_F}
    \begin{align*}
        \bftau : \bfxi \leq \frac{P}{2}(\Delta(\bfxi)-\xi_I)
    \end{align*}
    for all $\bftau \in \K(P)$. For $\bfxi\ne0$, the maximum is attained by
    \begin{align*}
        \bftau^* = \frac{P}{2}\left(\frac{\xi_I}{\Delta(\bfxi)}-1\right) \bfI + \frac{P}{e^2\Delta(\bfxi)}\dev\bfxi \in \K(P) \ .
    \end{align*}
    For $\bfxi=0$, every $\bftau\in\K(P)$ is a maximiser. For $P=0$ the
    only admissible stress is zero.
    Convexity, positive homogeneity and non-negativity (by \eqref{eq:xi_estimate}) follow from the properties of the support function of bounded sets with $0 \in \K(P)$. Lemma \ref{lema:prop_K} leads to the Lipschitz constant
    \begin{align*}
        \sup_{\bftau\in\K (P)} |\bftau|_F \leq c_\K P \ .
    \end{align*}
    The characterisation of zero dissipation holds since $D=0$ if and only if $\Delta(\bfxi) = \xi_I$, and from \eqref{eq:xi_estimate} follows $\dev \bfxi = 0$ and $\xi_I \geq 0$.
    The last estimate follows directly from
    \begin{align*}
        \Delta(\bfxi) - \xi_I \geq |\xi_I| - \xi_I = 2 (\xi_I)^- \ .
    \end{align*}
\end{proof}
Since the solutions of the time-discrete problem are a priori only in $L^3$, we
define the dissipation functional dually.
\begin{definition}\label{def:D}
    For $\bfv \in V$ let
    \begin{align}
        \D(\bfv) := \sup_{\bftau\in \Sigma} \left(- \int_\Omega \bfv \cdot \div \bftau \dd x \right) \in [0, \infty] \ .
        \label{eq:def_diss}
    \end{align}
\end{definition}
When the dependence on the strength is needed, we write $\Sigma[P]$ and
$\D[P]$ for the stress set and dissipation functional defined above.
\begin{lemma}\label{lem:hubler}
    For $\mu >0$, $P\geq 0$ and $\bfxi \in \mathbb S$, let $\Delta_\mu := \max (\Delta(\bfxi),\mu)$ and
    \begin{align*}
        \bftau_\mu (P,\bfxi) := \frac{P}{2} \left(\frac{\xi_I}{\Delta_\mu}-1\right) \bfI + \frac{P}{e^2\Delta_\mu}\dev\bfxi \ .
    \end{align*}
    Then the following statements hold:
    \begin{enumerate}
        \item $\bftau_\mu(P,\bfxi) \in \K(P)$ for all $\bfxi$. On $\{ \Delta (\bfxi)\geq \mu\}$,
        $\bftau_\mu$ lies on $\partial\K(P)$ and coincides with the
        maximiser of the support function.
        \item
        \begin{align}
            \bftau_\mu (P, \bfxi) : \bfxi = \frac{P}{2}\Big(\frac{\Delta(\bfxi)^2}{\Delta_\mu} - \xi_I \Big) \geq D(P, \bfxi) - \frac{P\mu}{8}
            \label{eq:hubler}
        \end{align}
        \item The function $\bftau_\mu(P,\cdot)$ is globally Lipschitz continuous on $\mathbb S$ with constant $C_e P/\mu$, where $C_e$ only depends on $e$.
        \item $\bftau_\mu$ is linear in $P$: we have $\bftau_\mu(P,\bfxi) = P\bftau_\mu (1,\bfxi)$, in particular $|\bftau_\mu(P, \bfxi)- \bftau_\mu(P',\bfxi)|_F \leq c_\K |P-P'|$ and $\bftau_\mu(0,\cdot) \equiv 0$.
    \end{enumerate}
\end{lemma}
\begin{proof}
    We calculate the invariants of $\bftau_\mu$
    \begin{align*}
        \tau_I = \frac{P}{2}\Big(\frac{\xi_I}{\Delta_\mu}-1\Big) \quad \text{and} \quad \tau_{II} = \frac{P}{2e^2\Delta_\mu}\xi_{II} \ .
    \end{align*}
    \textit{1.}
    For $P=0$ the stress is zero and the claim is immediate. For $P>0$
    we insert these invariants into \eqref{eq:def_F} and get
    \begin{align*}
        1\geq F(\bftau_\mu, P) +1 =\frac{\xi_I^2+e^{-2}\xi_{II}^2}{\Delta_\mu^2} = \frac{\Delta(\bfxi)^2}{\Delta_\mu^2} \ .
    \end{align*}
    For $\Delta \geq \mu$ we have $\Delta_\mu = \Delta$, the quotient
    is $1$ and $\bftau_\mu$ is the maximiser from the proof of Lemma \ref{lem:prop_D}.

    \textit{2.}
    A direct computation yields
    \begin{align*}
        \bftau_\mu : \bfxi = \frac{P}{2\Delta_\mu}(\xi_I^2+e^{-2}\xi_{II}^2) - \frac{P}{2}\xi_I \ .
    \end{align*}
    For $\Delta \geq \mu$ this coincides with $D(P,\bfxi)$.
    For $\Delta < \mu$, the gap satisfies
    \begin{align*}
        D(P, \bfxi) - \bftau_\mu : \bfxi = \frac{P}{2} \Big(\Delta -\frac{\Delta^2}{\mu}\Big) = \frac{P}{2}\Delta \Big(1-\frac{\Delta}{\mu}\Big) \leq \frac{P}{2} \frac{\mu}{4} = \frac{P\mu}{8} \ .
    \end{align*}
    \textit{3.}
    On the open set $\{\Delta < \mu\}$ we have $\Delta_\mu = \mu$ and $\bftau_\mu(P,\cdot)$ is affine with constant $C_e P/\mu$.
    On $\{\Delta >\mu\}$ we have $\bftau_\mu(P,\bfxi) = \frac{P}{2}\big(\Delta(\bfxi)^{-1}(\xi_I\bfI+2e^{-2}\dev\bfxi)-\bfI \big)$ and the
    quotient rule leads to $|\nabla \bftau_\mu| \leq C_e P/\Delta \leq C_e P/\mu$. Since $\bftau_\mu(P,\cdot)$ is continuous, $\{\Delta \leq \mu\}$ is convex and every segment $[\bfxi_1, \bfxi_2]$ cuts the interface in at most 2 points, a piecewise application of the fundamental theorem of calculus gives the global Lipschitz constant as the maximum of both constants.

    \textit{4.} $\bftau_\mu$ is linear in $P$ and the constant follows from \textit{1.}\ and Lemma \ref{lema:prop_K}: $|\bftau_\mu(1,\bfxi)|_F \leq c_\K$.

\end{proof}

\begin{lemma} \label{lem:bdd}
    Under Assumption \ref{ass}, the following statements hold:
    \begin{enumerate}
        \item $\D : V \rightarrow [0,\infty]$ is convex, positively homogeneous of degree 1, and weakly lower semicontinuous on $V$. Also $\D (0)=0$, so $\dom \D \neq \emptyset$.
        \item
        For $\bfv \in W^{1,3}(\Omega)^2$ with arbitrary trace, there is a relaxed
        boundary term
        \begin{align}
            \D (\bfv) = \int_\Omega D(P(x),\varepsilon(\bfv)(x)) \dd x + \int_{\Gamma_D} D(P(x), -(\bfv\odot \nu) ) \dd \mathcal{H}^1 \ , \label{eq:int_darstellung}
        \end{align}
        with $(\bfv \odot \nu) := \frac{1}{2}(\bfv \otimes \nu + \nu \otimes \bfv)$.
    \end{enumerate}
\end{lemma}
\begin{proof}
    \textit{1.}
    For any $\bftau \in \Sigma$ the map $\bfv \mapsto \int_\Omega \bfv \cdot \div \bftau \dd x$ is linear and (weakly) $L^3$-continuous, since $\div\bftau\in V^*$. Then $\D$, as the pointwise
    supremum over such functionals, is convex and weakly lower semicontinuous. For the homogeneity we notice that
    $\Sigma$ is independent of $\bfv$ and the functional is linear in $\bfv$. Thus homogeneity of degree 1 holds, since
    the supremum scales accordingly. Since $0 \in \Sigma$ and $\div 0 =0$, we have $\D \geq 0$ and $\D (0) =0$.

    \textit{2.}
    The inequality "$\leq$" is proved in Lemma \ref{lem:prove_leq} in the appendix
    by a slicing argument \cite{A83}. We prove the reverse inequality.
    Let $\mu, \rho, h, \eta >0$ and $k,j \in \mathbb N$.

    Step 1: Let $I := \bar \Gamma_D \cap \bar \Gamma_N$. From the disjointness of $\Gamma_D$ and $\Gamma_N$
    follows $\Gamma_D \cap I = \emptyset$. Choose $\zeta_\eta \in \mathrm{Lip}(\mathbb R^2)$ with
    $\zeta_\eta=0$ in $U_{\eta/2}(I):=\{x\in \mathbb R^2 :\dist(x,I)<\eta/2\}$ and $\zeta_\eta =1$
    in $\mathbb R^2 \setminus U_{\eta}(I)$, $0\leq \zeta_\eta \leq 1$. Additionally let
    $\vartheta_h(x):= \psi (\dist(x,\bar\Gamma_D)/h)$ with $\psi \in \mathrm{Lip}(\mathbb R)$ and
    $\psi = 0$ in $[1,\infty)$, $\psi =1$  in $(-\infty,1/2]$ and $0\leq \psi \leq 1$.
    Set $K_\eta := \bar\Gamma_D \setminus U_{\eta/4}(I)$, which is compact, and
    since $K_\eta \cap \bar\Gamma_N \subseteq I \setminus U_{\eta/4}(I) = \emptyset$ we have
    $d_\eta := \dist(K_\eta,\bar\Gamma_N)>0$, with $d_\eta = \infty$ for $K_\eta = \emptyset$.

    For $h<\min (\eta/4,d_\eta)$, $\vartheta_h\zeta_\eta$ vanishes on a neighbourhood of $\bar\Gamma_N$.
    Indeed, if $\vartheta_h(x)\zeta_\eta(x) \neq 0$, then $\dist(x,\bar\Gamma_D)<h$ and $\dist(x,I)\geq \eta/2$,
    and for the nearest point $y\in \bar\Gamma_D$ we have
    \begin{align*}
        \dist(y,I) \geq \eta/2 - h \geq \eta/4 \ ,
    \end{align*}
    hence $y\in K_\eta$ and $\dist(x,\bar\Gamma_N) \geq d_\eta-h >0$.
    Let $\chi_h \in C^\infty_c(\Omega)$ with $0\leq \chi_h\leq 1$, $\chi_h =1$ on $\{x:\dist(x,\partial\Omega)\geq 3h\}$,
    $\supp \chi_h \subset \{x : \dist(x,\partial\Omega)\geq 2h \} $.
    The supports of $\chi_h$ and $\vartheta_h$ are disjoint.

    Step 2: Let $g := - (\bfv \odot \nu) \in L^3(\Gamma_D; \mathbb S)$. On the compact metric space
    $(\bar\Gamma_D,|\cdot|)$ with the finite measure $\mathcal{H}^1$, Lipschitz functions are dense in $L^3$.
    Choose Lipschitz functions $G_j$ with $\Vert G_j - g\Vert_{L^3(\Gamma_D)}\rightarrow 0$ and
    extend them componentwise (McShane \cite{M34}, in a basis of $\mathbb S$) Lipschitz continuously to $\mathbb R^2$. Additionally let
    $\bfv_k \in C^\infty(\mathbb R^2; \mathbb R^2)$ with $\bfv_k \rightarrow \bfv$ in $W^{1,3}(\Omega)$.
    By the Tietze extension theorem and multiplication with a compactly supported continuous cutoff
    equal to $1$ on $\bar\Omega$, choose $\widetilde P\in C_c(\mathbb R^2)$ with
    $\widetilde P=P$ on $\bar\Omega$ and $0\leq\widetilde P\leq\Vert P\Vert_{C(\bar\Omega)}$.
    Let
    \begin{align*}
        \omega_{\widetilde P}(r):=\sup_{|x-y|\leq r}|\widetilde P(x)-\widetilde P(y)|,
        \qquad x,y\in\mathbb R^2,
    \end{align*}
    so that $\omega_{\widetilde P}(r)\rightarrow0$ as $r\downarrow0$ by uniform continuity.
    For the standard non-negative mollifier $\varphi_\rho$ supported in $B_\rho(0)$, define
    \begin{align*}
        P_\rho := (\widetilde P * \varphi_\rho - \omega_{\widetilde P}(\rho))^+ \ .
    \end{align*}
    Then $P_\rho$ is Lipschitz continuous and
    \begin{align}
        0\leq P_\rho\leq P\quad\text{on}\ \bar\Omega,
        \qquad \Vert P-P_\rho\Vert_{C(\bar\Omega)}\leq2\omega_{\widetilde P}(\rho) \ .
        \label{eq:Prho}
    \end{align}
    Indeed, $|\widetilde P*\varphi_\rho-P|\leq\omega_{\widetilde P}(\rho)$ on $\bar\Omega$.
    Subtracting this modulus and taking the positive part gives
    $(P-2\omega_{\widetilde P}(\rho))^+\leq P_\rho\leq P$ there.

Step 3:
Define
\begin{align*}
    \bftau &:= \chi_h \bftau_\mu (P_\rho, \varepsilon(\bfv_k)) + \vartheta_h\zeta_\eta\bftau_\mu(P_\rho,G_j) \\
    &=: \bftau^i + \bftau^b \ .
\end{align*}
Both terms are Lipschitz continuous by Lemma \ref{lem:hubler}, pointwise admissible,
\begin{align*}
    \bftau_\mu (P_\rho,\cdot) \in \K(P_\rho) \subseteq \K(P) \ ,
\end{align*}
by Lemma \ref{lem:hubler} and Lemma \ref{lema:prop_K}. Scaling with a factor in $[0,1]$ preserves admissibility via
convexity and $0\in \K$. Additionally their supports are disjoint, so that the sum is pointwise
admissible and $\div \bftau \in L^\infty\subset L^{3/2}$. For $\bfw \in C^1(\bar \Omega)$ with
$\bfw_{|\Gamma_D} =0$, by integration by parts we have
\begin{align}
    \int_\Omega \bftau : \varepsilon(\bfw) \dd x + \int_\Omega \bfw \cdot \div \bftau\dd x
    = \int_{\partial\Omega} \bfw \cdot (\bftau \nu) \dd \mathcal{H}^1 = 0 \label{eq:ip}
\end{align}
since $\bfw$ vanishes on $\Gamma_D$ and $\bftau$ vanishes on $\bar\Gamma_N$ ($\bftau^i$ is compactly
supported and $\bftau^b$ because of Step 1 for $h<\min(\eta/4,d_\eta)$), and the remaining boundary set has measure zero
with $\bftau\nu \in L^\infty$. Thus $\bftau \in \Sigma$.

Step 4:
    Since $\bftau$ is Lipschitz continuous, approximation by $C^1(\bar\Omega)$ functions gives
    the Green formula for every $\bfv\in W^{1,3}(\Omega)$ with arbitrary trace:
    \begin{align*}
        -\int_\Omega\bfv\cdot\div\bftau\dd x
        =\int_\Omega\bftau:\varepsilon(\bfv)\dd x
        -\int_{\partial\Omega}\bfv\cdot(\bftau\nu)\dd\mathcal{H}^1 \ .
    \end{align*}
    With $\bfv\cdot(\bftau\nu)=\bftau:(\bfv\odot\nu)$, $\vartheta_h=1$ on $\bar\Gamma_D$
    and $\bftau=0$ on $\bar\Gamma_N$, we therefore obtain
    \begin{align*}
        \D(\bfv)&\geq-\int_\Omega\bfv\cdot\div\bftau\dd x\\
        &=\int_\Omega\chi_h\bftau_\mu(P_\rho,\varepsilon(\bfv_k)):\varepsilon(\bfv)\dd x
        +R_h+\int_{\Gamma_D}\zeta_\eta\bftau_\mu(P_\rho,G_j):g\dd\mathcal{H}^1,
    \end{align*}
    where $R_h:=\int_\Omega\bftau^b:\varepsilon(\bfv)\dd x$ is the volume contribution of the boundary field.
    Since $|\bftau^b|_F\leq c_\K\Vert P\Vert_{L^\infty}$ and $\bftau^b=0$ whenever
    $\dist(x,\bar\Gamma_D)\geq h$, it satisfies
    \begin{align*}
        |R_h|\leq r(h):=c_\K\Vert P\Vert_{L^\infty}
        \int_{\Omega\cap U_h(\bar\Gamma_D)}|\varepsilon(\bfv)|_F\dd x\rightarrow0
        \quad\text{as }h\downarrow0 \ .
    \end{align*}
    This follows by dominated convergence from $\varepsilon(\bfv)\in L^1(\Omega)$ and holds
    uniformly in the other approximation parameters. In the lower bound we retain the error $-r(h)$.

Step 5:
    Lemma \ref{lem:hubler}, H\"older and $P_\rho \leq \Vert P\Vert_{L^\infty}$ lead to
    \begin{align*}
        &\left|\int_\Omega \chi_h[\bftau_\mu(P_\rho, \varepsilon(\bfv_k)) - \bftau_\mu(P_\rho,\varepsilon(\bfv))]:\varepsilon(\bfv)\dd x\right|\\
        &\qquad\leq \frac{C_e \Vert P\Vert_{L^\infty(\Omega)}}{\mu} \Vert \varepsilon(\bfv_k)-\varepsilon(\bfv)\Vert_{L^{3/2}}
        \Vert \varepsilon(\bfv)\Vert_{L^3} \rightarrow 0
    \end{align*}
    as $k\rightarrow \infty$ and
    \begin{align*}
        \int_\Omega \chi_h\bftau_\mu (P_\rho, \varepsilon(\bfv)):\varepsilon(\bfv)\dd x \geq
        \int_\Omega \chi_h D(P_\rho, \varepsilon(\bfv))\dd x - \frac{\mu}{8} \Vert P \Vert_{L^1} \ .
    \end{align*}
    We apply the linearity of $D(\cdot,\bfxi)$ in the strength from Lemma \ref{lem:prop_D}, \eqref{eq:Prho} and
    $\Delta (\bfxi) - \xi_I \leq 2\Delta(\bfxi) \leq 2 \sqrt{2}|\bfxi|_F$ and conclude
    \begin{align*}
        \int_\Omega \chi_h|D(P,\varepsilon(\bfv))-D(P_\rho, \varepsilon(\bfv))|\dd x
        \leq 2 \sqrt{2}\omega_{\widetilde P}(\rho) \Vert\varepsilon(\bfv)\Vert_{L^{1}} \rightarrow 0
    \end{align*}
    as $\rho \rightarrow 0$.

Step 6:
    Pointwise a.e.\ on $\Gamma_D$ we have \eqref{eq:hubler}.  Lemma \ref{lem:hubler}
    and Lemma \ref{lema:prop_K}  lead to the $L^{\infty}$ bound
    $|\bftau_\mu (P_\rho, \cdot)|_F \leq c_\K \Vert P \Vert_{L^\infty} $. Applying these
    estimates together with the Lipschitz continuity from Lemma \ref{lem:prop_D} and the $P$-linearity
    we get the following chain of estimates
    \begin{align*}
        \bftau_\mu(P_\rho, G_j) :g &=  \bftau_\mu(P_\rho, G_j) : G_j + \bftau_\mu (P_\rho, G_j) :(g-G_j) \\
        &\geq D(P_\rho, G_j) - \frac{P_\rho\mu}{8} - c_\K \Vert P \Vert_{L^\infty} |g-G_j|_F \\
        &\geq D(P_\rho, g) - \frac{P_\rho\mu}{8} - 2c_\K \Vert P \Vert_{L^\infty} |g-G_j|_F \\
        &\geq D(P, g) - \frac{P_\rho\mu}{8} - 2c_\K \Vert P \Vert_{L^\infty} |g-G_j|_F - 2 \sqrt{2} \omega_{\widetilde P}(\rho)|g|_F \ .
    \end{align*}
    Multiply with $\zeta_\eta \in [0,1]$ (keep it in the main term, estimate it by $1$ in the error terms), integrate
    and estimate:
    \begin{align}
        \begin{split}
        \int_{\Gamma_D} \zeta_\eta \bftau_\mu (P_\rho, G_j) : g \dd \mathcal{H}^{1}
        &\geq \int_{\Gamma_D} \zeta_\eta D(P, g) \dd \mathcal{H}^{1}
        -\frac{\mu}{8} \Vert P \Vert_{L^{1}(\Gamma_D)}\\
        &\quad - C \Vert g-G_j\Vert_{L^{3}(\Gamma_D)} - C' \omega_{\widetilde P}(\rho) \ .
        \end{split}
    \end{align}

Step 7:

All limits are taken with the left-hand side $\D(\bfv)$ fixed. First we take the limits $k\rightarrow \infty$,
$j\rightarrow \infty$ and $\rho\rightarrow 0$. Combining Steps 4--6 yields
\begin{align*}
    \D(\bfv)\geq{}&\int_\Omega\chi_hD(P,\varepsilon(\bfv))\dd x
    +\int_{\Gamma_D}\zeta_\eta D(P,g)\dd\mathcal{H}^1\\
    &-\frac{\mu}{8}\left(\Vert P\Vert_{L^1(\Omega)}+\Vert P\Vert_{L^1(\Gamma_D)}\right)-r(h) \ .
\end{align*}
Then let $h\rightarrow0$ with fixed $\eta$. Since $0\leq\chi_h\leq1$ and $\chi_h\rightarrow1$ pointwise in $\Omega$,
dominated convergence applies to the volume term, while $r(h)\rightarrow0$ by Step 4.
The restriction $h<\min(\eta/4,d_\eta)$ is preserved for all sufficiently small $h$.
Then $\eta \rightarrow 0$: since $\Gamma_D \cap I = \emptyset$ (Step 1), we have $\zeta_\eta \rightarrow 1$ pointwise
on $\Gamma_D$, with the majorant $c_\K \Vert P \Vert_{L^{\infty}}|g|\in L^{1}(\Gamma_D)$. Finally $\mu \rightarrow 0$,
and both remaining error terms $\frac{\mu}{8}(\Vert P\Vert_{L^{1}(\Omega)}+\Vert P\Vert_{L^{1}(\Gamma_D)})$ vanish.
\end{proof}

\subsection{Stress reconstruction}
\begin{lemma}\label{lem:stress_reconstruction}
    Under Assumption \ref{ass}, let $\bfu\in V$ with $\D(\bfu)<\infty$ and
    $\bfxi\in V^*$. Then $\bfxi\in\partial\D(\bfu)$ if and only if there is
    a $\bfsigma\in\Sigma$ such that
    \begin{align}
        \bfxi=-\div\bfsigma,\qquad
        \D(\bfu)=-\int_\Omega\bfu\cdot\div\bfsigma\dd x \ .
        \label{eq:stress_saturation}
    \end{align}
\end{lemma}
\begin{proof}
    Set $C:=\{-\div\bftau:\bftau\in\Sigma\}\subset V^*$. This set is
    nonempty and convex. It is also norm closed. Indeed, if
    $-\div\bftau_j\rightarrow\bfxi$ in $V^*$, then
    $|\bftau_j|_F\leq c_\K P$ gives a subsequence converging weakly-$\ast$
    in $L^\infty$ to a symmetric field $\bfsigma$ by weak-$\ast$ compactness
    \cite[Chapter 3]{Br11}. The pointwise constraint
    $\bfsigma(x)\in\K(P(x))$ is preserved because it is closed and convex.
    Distributional testing gives $-\div\bfsigma=\bfxi$. Passing to the limit
    in the Green identity against $C^1(\bar\Omega)$ fields with zero
    Dirichlet trace preserves the weak Neumann condition. Thus
    $\bfsigma\in\Sigma$ and $\bfxi\in C$.

    By definition, $\D$ is the support function of $C$. Since $V$ is
    reflexive and $C$ is closed and convex, its Fenchel conjugate is the
    indicator function of $C$. By the conjugacy and subdifferential relations
    in \cite[Chapter I]{ET99}, Fenchel equality characterises
    $\bfxi\in\partial\D(\bfu)$ by $\bfxi\in C$ and
    $\langle\bfxi,\bfu\rangle=\D(\bfu)$, which is precisely
    \eqref{eq:stress_saturation}.
\end{proof}

\begin{remark}
    Equation \eqref{eq:stress_saturation} is the global saturation condition
    for velocities in $L^3$. If additionally
    $\bfu\in W^{1,3}(\Omega)^2$ has zero Dirichlet trace, the Green
    identity and \eqref{eq:int_darstellung} imply
    $\bfsigma:\varepsilon(\bfu)=D(P,\varepsilon(\bfu))$ a.e.\ Hence the
    pointwise normal flow rule holds in that regularity class. For general
    $L^3$ velocities, \eqref{eq:stress_saturation} is the relaxed formulation.
    The stress need not be unique: for full Dirichlet boundary, constant
    $P>0$ and zero velocity, both $0$ and $-P\bfI/2$ are admissible,
    divergence-free stresses satisfying the saturation condition.
\end{remark}

\subsection{Measure structure of finite dissipation}
Finite dissipation does not control the full deformation: isotropic
expansion has zero dissipation. Nevertheless, under spatial Lipschitz
regularity of $P$, suitable stress tests yield local measure structure and
a global bound on the negative part of the weighted divergence. This
additional regularity assumption is used only in the following theorem.

\begin{theorem}\label{th:measure_structure}
    Let $P\in C^{0,1}(\bar\Omega)$, $P\geq0$, and let $\bfv\in V$ with
    $\D(\bfv)<\infty$. Define the distributions
    \begin{align*}
        S&:=\div(P\bfv)-\bfv\cdot\nabla P,\\
        T_M&:=\div(PM\bfv)-(M\bfv)\cdot\nabla P
        \quad\text{for constant }M\in\mathbb S \ .
    \end{align*}
    Then the following statements hold.
    \begin{enumerate}
        \item[(a)] $S$ is a locally finite signed Radon measure in $\Omega$,
        and its negative part satisfies the sharp global bound
        \begin{align}
            S^-(\Omega)\leq\D(\bfv) \ .\label{eq:S-}
        \end{align}
        \item[(b)] For every deviatoric $M$, $T_M$ is a locally finite signed
        Radon measure. Set $c_M:=|M|_F/\sqrt{2}$, the stress II-invariant of
        $M$. If $\omega\Subset\Omega$ is open and
        $\chi\in C_c^\infty(\Omega)$ satisfies $0\leq\chi\leq1$ and
        $\chi=1$ on a neighbourhood of $\bar\omega$, then
        \begin{align}
            |T_M|(\omega)
            &\leq e c_M\bigl(2\D(\bfv)+\langle S,\chi\rangle\bigr)
            \notag\\
            &\leq e c_M\bigl(2\D(\bfv)+S^+(\supp\chi)\bigr) \ .
            \label{eq:TM}
        \end{align}
        For $M=0$, both sides are zero.
        \item[(c)] The symmetric tensor-valued distribution
        \begin{align}
            A:=\varepsilon(P\bfv)-\bfv\odot\nabla P\,\mathcal{L}^2
            \label{eq:weighted_deformation}
        \end{align}
        is a locally finite Radon measure. We use $P\varepsilon(\bfv)$ as
        shorthand for this distribution. Moreover,
        \begin{align*}
            P\bfv\in BD_{\mathrm{loc}}(\Omega),\qquad
            \bfv\in BD_{\mathrm{loc}}(\{P>0\}),
        \end{align*}
        and hence $\bfv\in BD_{\mathrm{loc}}(\{P>\gamma\})$ for every
        $\gamma>0$. Here $BD$ denotes the space of integrable vector fields
        whose distributional symmetric gradient is a finite Radon measure.
    \end{enumerate}
\end{theorem}
\begin{proof}
    The distributions are well defined because $P\bfv$ and
    $\bfv\cdot\nabla P$ belong to $L^3$. For scalar
    $\mu\in C_c^\infty(\Omega)$,
    \begin{align}
        \langle S,\mu\rangle
        =-\int_\Omega\bfv\cdot\nabla(P\mu)\dd x \ .
        \label{eq:ip_measure}
    \end{align}

    Step 1: If $0\leq\mu\leq1$, the compactly supported Lipschitz field
    $\bftau=-P\mu\bfI$ lies in $\Sigma$. Indeed, for $P>0$ its normalised
    yield expression is $(1-2\mu)^2\leq1$, and for $P=0$ it is zero.
    Consequently,
    \begin{align}
        -\langle S,\mu\rangle
        =-\int_\Omega\bfv\cdot\div\bftau\dd x\leq\D(\bfv) \ .
        \label{eq:S}
    \end{align}
    Fix a compact $K\Subset\Omega$ and choose
    $\chi\in C_c^\infty(\Omega)$ with $0\leq\chi\leq1$ and $\chi=1$
    near $K$. For $\varphi\in C_c^\infty(\Omega)$ with
    $\supp\varphi\subset K$ and $\Vert\varphi\Vert_{L^\infty}\leq1$, both
    $(\chi+\varphi)/2$ and $(\chi-\varphi)/2$ are admissible in
    \eqref{eq:S}. Therefore
    \begin{align*}
        |\langle S,\varphi\rangle|
        \leq2\D(\bfv)+\langle S,\chi\rangle
        \leq2\D(\bfv)+|\langle S,\chi\rangle| \ .
    \end{align*}
    Thus $S$ is a distribution of order zero on every compact subset and
    hence a locally finite signed Radon measure. The characterisation
    \begin{align*}
        S^-(\Omega)=\sup\{-\langle S,\mu\rangle:
        \mu\in C_c^\infty(\Omega),\ 0\leq\mu\leq1\}
    \end{align*}
    proves \eqref{eq:S-}. To see sharpness, choose nonzero
    $0\leq P\in C_c^\infty(\Omega)$ and $\bfv(x)=-x$. Then
    $\varepsilon(\bfv)=-\bfI$, $S^-=(P\div\bfv)^-\mathcal{L}^2
    =2P\mathcal{L}^2$, and \eqref{eq:int_darstellung}, whose boundary term
    vanishes here, gives $\D(\bfv)=2\int_\Omega P\dd x$.

    Step 2: Let $M\ne0$ be deviatoric and put $\widehat M=M/c_M$. For
    $\varphi\in C_c^\infty(\omega)$ with $|\varphi|\leq1$, define
    \begin{align*}
        \bftau_\pm=-\frac{P\chi}{2}\bfI
                    \ \pm\frac{P\varphi}{2e}\widehat M \ .
    \end{align*}
    We have $|\varphi|\leq\chi$, so for $P>0$
    \begin{align*}
        F(\bftau_\pm,P)+1=(1-\chi)^2+\varphi^2
        \leq(1-\chi)^2+\chi^2\leq1 \ .
    \end{align*}
    For $P=0$, both stresses are zero. They are Lipschitz and compactly
    supported, hence belong to $\Sigma$. Testing the definition of $\D$
    with the two signs yields
    \begin{align*}
        -\frac{1}{2}\langle S,\chi\rangle
        \ \pm\frac{1}{2e c_M}\langle T_M,\varphi\rangle
        \leq\D(\bfv) \ .
    \end{align*}
    Therefore
    \begin{align*}
        |\langle T_M,\varphi\rangle|
        \leq e c_M\bigl(2\D(\bfv)+\langle S,\chi\rangle\bigr) \ .
    \end{align*}
    This uniform estimate proves the local measure structure of $T_M$ and,
    on taking the supremum over $\varphi$, the first inequality in
    \eqref{eq:TM}. The second follows from $0\leq\chi\leq1$.
    The case $M=0$ is immediate.

    Step 3: By the distributional definitions, $\tr A=S$ and $M:A=T_M$
    for every constant deviatoric $M$. Two linearly independent deviatoric
    matrices together with $\bfI$ determine all components of $A$, so
    $A$ is a locally finite matrix-valued measure. As
    $\bfv\odot\nabla P\in L^3$, \eqref{eq:weighted_deformation} gives
    $P\bfv\in BD_{\mathrm{loc}}(\Omega)$. On every compact subset of
    $\{P>0\}$, the factor $P^{-1}$ is bounded and Lipschitz. The product
    rule for a Lipschitz function and a $BD$ field gives
    \begin{align*}
        \varepsilon(\bfv)
        =\varepsilon\bigl(P^{-1}(P\bfv)\bigr)=P^{-1}A,
    \end{align*}
    since the terms involving $\nabla P$ cancel. Thus
    $\bfv\in BD_{\mathrm{loc}}(\{P>0\})$. See also the $BD$ framework
    in \cite{RG80,KT83}.
\end{proof}

\section{Well-posedness of the time-discrete problem}\label{sec:td}
We define the potential of the local part,
\begin{align}
    \Phi(x,\bfv) := \frac{m(x)}{2k} |\bfv - \bfu^{n-1}(x)|^2  + \frac{c_o}{3} |\bfu_o (x) -\bfv|^3 - \bff(x)\cdot \bfv, \quad\text{for all }\bfv\in\mathbb R^2\ ,
\end{align}
and the (non-variational) operator $\mathcal{A} : V \rightarrow V^{*}$
\begin{align*}
    \mathcal{A} (\bfu) &:= \frac m k (\bfu -\bfu^{n-1}) + m f_c \bfe_3 \times \bfu + c_o |\bfu-\bfu_o|(\bfu-\bfu_o)\\
    &= \partial_v \Phi(\cdot,\bfu) + mf_c\bfe_3 \times \bfu +\bff \ .
\end{align*}
The weak form of \eqref{eq:time_discrete_model} is: find $\bfu \in V$ such that
\begin{align}
    \langle \mathcal{A}(\bfu)-\bff, \bfv-\bfu\rangle + \D(\bfv) -\D(\bfu) \geq 0 \quad \text{for all}\ \bfv\in V \ , \label{eq:vi_td}
\end{align}
an elliptic VI of the second kind. Since $\D$ is convex and $\mathcal{A}$ single-valued, it is equivalent
to the subdifferential inclusion
\begin{align}
    \bff - \mathcal{A} (\bfu) \in \partial\D(\bfu) \quad \text{in} \ V^{*} \ .
    \label{eq:sub_vi_td}
\end{align}
We first state a monotonicity estimate for the ocean drag.
\begin{lemma}\label{lem:drag}
    For any $a,b\in\mathbb R^{n}$ we have
    \begin{align}
        (|a|a-|b|b) \cdot (a-b) \geq \frac{1}{2} (|a|+|b|)|a-b|^{2} \geq \frac 1 2 |a-b|^{3} \ . \label{eq:ocean_drag}
    \end{align}
    Both inequalities are sharp.
\end{lemma}
\begin{proof}
    Write $a\cdot b = |a||b|\cos(\theta)$, then we have
 \begin{align*}
      (|a|a-|b|b) \cdot (a-b) &= |a|^{3} + |b|^{3} -(|a|+|b|)|a||b|\cos(\theta) \\
      &= (|a| + |b|) (|a| - |b|)^{2} +(|a|+|b|)|a||b|(1-\cos(\theta)) \ .
 \end{align*}
 Additionally $|a-b|^2 = (|a|-|b|)^2+2|a||b|(1-\cos\theta)$, so
 \begin{align*}
    &(|a|+|b|)(|a|-|b|)^2 + (|a|+|b|)|a||b|(1-\cos\theta)\\
    &\qquad\geq \frac 1 2
    (|a|+|b|)\left((|a|-|b|)^2+2|a||b|(1-\cos\theta)\right)\\
    &\qquad= \frac 1 2 (|a|+|b|)|a-b|^{2} \ .
 \end{align*}
 The second inequality in \eqref{eq:ocean_drag} follows by the triangle inequality. The first inequality is sharp for $|a|=|b|$ and the second one
 for $b=-a$.
\end{proof}
\section{Existence, uniqueness and stability}
The following theorem establishes well-posedness of the time-discrete
problem. Both $P$ and $m$ may vanish, and $\Gamma_D=\emptyset$ is allowed.
Coercivity is provided solely by the ocean drag.
\begin{theorem}\label{th:td}
    Under Assumption \ref{ass}, the variational inequality \eqref{eq:vi_td} has a unique solution $\bfu \in V$. The following statements hold:
    \begin{enumerate}
        \item Let $\bfu_1, \bfu_2$ be the solutions corresponding to $\bff_1,\bff_2\in V^*$. Then
        \begin{align*}
            \Vert \bfu_1 -\bfu_2\Vert_{L^3}^2 \leq \frac{2}{c_o} \Vert \bff_1-\bff_2\Vert_{L^{3/2}} \ .
        \end{align*}
        Thus the solution map $\bff \mapsto \bfu$ is H\"older continuous with exponent $\frac 1 2$.
    \item In the interior of $\{P=0\}$, \eqref{eq:sub_vi_td} reduces to the pointwise algebraic
    equation
    \begin{align*}
        \frac m k (\bfu - \bfu^{n-1}) + m f_c \bfe_3 \times \bfu + c_o |\bfu-\bfu_o|(\bfu-\bfu_o) = \bff
    \end{align*}
    which is uniquely solvable.
    \item Let $\bfu_1,\bfu_2$ be the solutions corresponding to the ocean
    velocities $\bfu_o^{(1)},\bfu_o^{(2)}$, with all other data fixed. Set
    $\bfeta:=\bfu_o^{(1)}-\bfu_o^{(2)}$ and $\bfz_i:=\bfu_i-\bfu_o^{(i)}$.
    Then
    \begin{align}
        \Vert\bfz_1-\bfz_2\Vert_{L^{3}}^3 &\leq
        4 (\Vert\bfz_1\Vert_{L^3}+\Vert\bfz_2\Vert_{L^3})\Vert\bfeta\Vert_{L^3}^2\label{eq:first_uo}\\
        \frac{1}{k}\Vert\sqrt{m}(\bfu_1-\bfu_2)\Vert^2_{L^2}
        &\leq c_o(\Vert\bfz_1\Vert_{L^3}+\Vert\bfz_2\Vert_{L^3})\Vert\bfeta\Vert^2_{L^3}\label{eq:second_uo}\\
        \Vert\bfu_1-\bfu_2\Vert_{L^3}&\leq \Vert\bfeta\Vert_{L^3} +(4(\Vert\bfz_1\Vert_{L^3}+\Vert\bfz_2\Vert_{L^3}))^{1/3}\Vert\bfeta\Vert_{L^3}^{2/3} \ . \label{eq:third_uo}
    \end{align}
    On bounded subsets of the ocean-velocity data space $L^3$, the solution
    map $\bfu_o\mapsto\bfu$ is H\"older continuous with exponent $2/3$ in
    $L^3$ and Lipschitz continuous with respect to the mass seminorm
    $\Vert\sqrt m\,\cdot\Vert_{L^2}$, with the other data fixed.
    In the interior of $\{P=0\}$ we have pointwise a.e.
    \begin{align*}
        |\bfu_1-\bfu_2|\leq \sqrt{1+(kf_c)^2}|\bfu_o^{(1)}-\bfu_o^{(2)}| \ .
    \end{align*}
    \end{enumerate}
\end{theorem}
\begin{proof}
    We verify the assumptions of the existence theorem for elliptic variational inequalities
    of the second kind with monotone main term \cite[Theorem 54.A]{Z85}. We recall that $V$ is reflexive and separable, and according to Lemma \ref{lem:bdd} the functional $\D:V\rightarrow [0,\infty]$ is convex, lower semicontinuous and proper. We show that $\mathcal{A} : V \rightarrow V^{*}$
    is bounded, hemicontinuous, monotone and coercive.

    Step 1:
    Pointwise we can estimate
    \begin{align*}
        |\mathcal{A}(\bfv)| \leq \frac{\Vert m\Vert_{L^{\infty}}}{k} |\bfv| + \frac{1}{k} \sqrt{\Vert m\Vert_{L^{\infty}}} \sqrt{m} |\bfu^{n-1}| + \Vert m f_c\Vert_{L^{\infty}} |\bfv| +c_o (|\bfv|+|\bfu_o|)^{2} \ .
    \end{align*}
    Since $\Omega$ is a bounded domain, we have the embeddings $L^{3} \hookrightarrow L^{3/2}$ and $L^{2} \hookrightarrow L^{3/2}$. Together with Assumption \ref{ass} (A4) we have
    $\mathcal{A}(\bfv)\in L^{3/2}(\Omega)^2=V^*$ with $\Vert \mathcal{A}(\bfv)\Vert_{V^{*}} \leq C (1+\Vert \bfv\Vert^{2}_V)$, so $\mathcal{A}$ is bounded.

    Step 2: We show for $\bfu, \bfv, \bfw \in V$ that the function $t \mapsto \langle \mathcal{A}(\bfu +t \bfw),\bfv\rangle$ is continuous on $[0,1]$. The integrands are continuous in $t$ and we have an integrable majorant independent of $t$:
    \begin{align*}
        C\big((|\bfu| + |\bfw| +|\bfu_o|+1)^{2} + \sqrt{m}|\bfu^{n-1}|/k \big) |\bfv| \ .
    \end{align*}
    By Lebesgue's theorem we conclude the hemicontinuity.

    Step 3:
    Let $\bfv, \bfw\in V$, $\delta := \bfv-\bfw$. The mass term gives
    \begin{align*}
        \int_\Omega \frac m k |\delta|^{2} \dd x \geq 0 \ .
    \end{align*}
    The Coriolis term is pointwise skew-symmetric, $(mf_c\bfe_3 \times \delta) \cdot \delta =0$, and
    Lemma \ref{lem:drag} leads to the estimate
    \begin{align*}
        \langle \mathcal{A}(\bfv)-\mathcal{A}(\bfw), \delta\rangle \geq \frac{c_o}{2} \int_\Omega
        (|\bfv-\bfu_o|+|\bfw-\bfu_o|)|\delta|^2 \dd x \geq \frac{c_o}{2} \Vert\delta\Vert_{L^{3}}^3
      \end{align*}
    for the ocean drag. Therefore $\mathcal{A}$ is strictly monotone.

    Step 4: To prove the coercivity we estimate with Young's inequality for $\bfv\in V$ and
    $d := \bfv -\bfu_o$:
    \begin{align*}
        &\langle \mathcal{A} (\bfv) - \bff , \bfv \rangle\\
        &\quad= \int_\Omega \left(\frac m k (\bfv-\bfu^{n-1})\cdot \bfv + c_o |d| d\cdot(d+\bfu_o)-\bff \cdot \bfv\right) \dd x\\
        &\quad\geq \int_\Omega \left(\frac{m}{2k} |\bfv|^{2} - \frac{m}{2k}|\bfu^{n-1}|^{2} + c_o |d|^3 - c_o |d|^2|\bfu_o|
        -|\bff|(|d|+|\bfu_o|)\right) \dd x\\
        &\quad\geq \frac{c_o}{3} \Vert d \Vert^3_{L^3} - \frac{1}{2k} \Vert \sqrt{m}\bfu^{n-1}\Vert^2_{L^2}
        - \frac{4c_o}{3} \Vert \bfu_o\Vert_{L^3}^3\\
        &\qquad - \frac{2}{3\sqrt{c_o}} \Vert \bff \Vert_{L^{3/2}}^{3/2}
        -\Vert\bff\Vert_{L^{3/2}} \Vert \bfu_o\Vert_{L^3} \ ,
    \end{align*}
    where we applied the Young inequalities $c_o|d|^2|\bfu_o| \leq \frac{c_o}{3}|d|^3 + \frac{4c_o}{3}|\bfu_o|^3$ and
    $|\bff||d| \leq \frac{c_o}{3}|d|^3 + \frac{2}{3\sqrt{c_o}}|\bff|^{3/2}$.
    The triangle inequality $\Vert \bfv\Vert^3_{L^3} \leq 4(\Vert d\Vert_{L^3}^3 + \Vert\bfu_o\Vert^3_{L^3})$
    leads to
    \begin{align*}
        \langle \mathcal{A}(\bfv) - \bff, \bfv - 0\rangle \geq \frac{c_o}{12}\Vert\bfv\Vert^3_{L^3} - C_0
    \end{align*}
    with $C_0 = C_0 (\Vert \sqrt{m}\bfu^{n-1}\Vert_{L^2}, \Vert\bfu_o\Vert_{L^3}, \Vert\bff\Vert_{L^{3/2}}, k, c_o)$.
    This shows the coercivity of $\mathcal{A}$.

    Step 5: The result in \cite{Z85} gives existence, and strict monotonicity gives uniqueness.

    Step 6: For the stability, we take $\bff_1, \bff_2$ and their solutions $\bfu_1, \bfu_2$.
    We test the variational inequality for $\bfu_1$ with $\bfv=\bfu_2$ and conversely, and add the two inequalities.
    This gives the stability estimate:
    \begin{align*}
        \frac{c_o}{2} \Vert \bfu_1-\bfu_2\Vert_{L^{3}}^3 \leq \langle\bff_1-\bff_2,\bfu_1-\bfu_2\rangle
        \leq \Vert \bff_1-\bff_2\Vert_{L^{3/2}}\Vert\bfu_1-\bfu_2\Vert_{L^3} \ .
    \end{align*}

    Step 7: For $\bftau\in\Sigma$ we have $|\bftau|\leq c_\K P=0$ a.e.\ in $\{P=0\}$ by Lemma \ref{lema:prop_K}.
    By Lemma \ref{lem:stress_reconstruction}, the subgradient
    $\bff-\mathcal{A}(\bfu)$ equals $-\div\bfsigma$ for some $\bfsigma\in\Sigma$.
    On the open set $\operatorname{int}\{P=0\}$ this stress vanishes, so its
    divergence is zero there in distributions and hence a.e.\ This gives
    the stated pointwise algebraic equation. For fixed $x$ the map $\bfv \mapsto \frac m k \bfv + m f_c\bfe_3\times \bfv
    +c_o |\bfv-\bfu_o|(\bfv-\bfu_o)$ is by Lemma \ref{lem:drag} strictly monotone, continuous and coercive on
    $\mathbb R^2$, hence bijective.

    Proof of \textit{3.}: Set $\delta:=\bfu_1-\bfu_2$. Testing as in Step 6 eliminates
    $\D$ and $\bff$. The Coriolis term vanishes and we have
    \begin{align*}
        \frac{1}{k}\Vert\sqrt{m}\delta\Vert^2_{L^2} +
         \langle c_o|\bfz_1|\bfz_1-c_o|\bfz_2|\bfz_2,\delta\rangle\leq 0 \ .
    \end{align*}
    We notice $\delta = (\bfz_1-\bfz_2)+\bfeta$. With $\omega := |\bfz_1|+|\bfz_2|$ we apply the first estimate in \eqref{eq:ocean_drag}
    \begin{align*}
        \langle c_o|\bfz_1|\bfz_1-c_o|\bfz_2|\bfz_2,\bfz_1-\bfz_2\rangle \geq \frac{c_o}{2}\int_\Omega \omega |\bfz_1-\bfz_2|^2\dd x
    \end{align*}
    and from the identity $|a|a-|b|b = |a|(a-b)+(|a|-|b|)b$ we conclude the pointwise estimate
    \begin{align*}
        \left|c_o|\bfz_1|\bfz_1-c_o|\bfz_2|\bfz_2\right|\leq c_o \omega |\bfz_1-\bfz_2| \ .
    \end{align*}
    Young's inequality, $c_o \omega |\bfz_1-\bfz_2||\bfeta|\leq \frac{c_o}{4}\omega |\bfz_1-\bfz_2|^2+c_o\omega|\bfeta|^2$ and H\"older lead to
    \begin{align*}
        \frac{1}{k}\Vert\sqrt{m}\delta\Vert^2_{L^2} + \frac{c_o}{4} \int_\Omega\omega |\bfz_1-\bfz_2|^2\dd x
        &\leq c_o \int_\Omega \omega|\bfeta|^2 \dd x\\
        &\leq c_o(\Vert\bfz_1\Vert_{L^3}+\Vert\bfz_2\Vert_{L^3})\Vert\bfeta\Vert^2_{L^3} \ .
    \end{align*}
    Since $|\bfz_1-\bfz_2|\leq \omega$ we have $\int_\Omega\omega|\bfz_1-\bfz_2|^2\dd x \geq \Vert\bfz_1-\bfz_2\Vert^3_{L^3}$, so we conclude \eqref{eq:first_uo} and \eqref{eq:second_uo}. The estimate \eqref{eq:third_uo} follows from $\delta = \bfz_1-\bfz_2 + \bfeta$.

    The coercivity estimate, tested at the solutions, bounds
    $\Vert\bfz_i\Vert_{L^3}$ uniformly when the ocean velocities range over
    a bounded subset of $L^3$ and the other data are fixed. This justifies
    the local continuity assertions in the theorem.

    In the interior of $\{P=0\}$ both solutions satisfy the pointwise equation in \textit{2.}. Subtracting the two equations gives
    \begin{align*}
        \frac{m}{k}\delta +mf_c\bfe_3\times \delta + c_o|\bfz_1|\bfz_1-c_o|\bfz_2|\bfz_2 =0 \ .
    \end{align*}
    Taking the scalar product with $\bfz_1-\bfz_2=\delta-\bfeta$, using monotonicity $(c_o|\bfz_1|\bfz_1-c_o|\bfz_2|\bfz_2)\cdot(\bfz_1-\bfz_2)\geq 0$ and $(\bfe_3\times\delta)\cdot\delta=0$, we obtain
    \begin{align*}
        \frac{m}{k}|\delta|^2 \leq \frac{m}{k}(\delta+kf_c\bfe_3\times\delta)\cdot\bfeta
        \leq \frac{m}{k}\sqrt{1+(kf_c)^2}|\delta||\bfeta|
    \end{align*}
    since $|\bfe_3\times\delta|=|\delta|$. This proves the claim where $m>0$. Where $m=0$,
    the equation reduces to $c_o|\bfz_1|\bfz_1=c_o|\bfz_2|\bfz_2=\bff$, therefore $\bfz_1=\bfz_2$ and $\delta = \bfeta$.
\end{proof}

\subsection{Discrete energy estimate}
\begin{lemma}\label{lem:DES}
    Let $\bfu^{n}\in V$ solve \eqref{eq:vi_td} at step $n$, with data
    $(m^n, P^n, \bff^{n}, \bfu^{n}_o)$ satisfying Assumption \ref{ass}, with initial value $\bfu^{n-1}$,
    $\sqrt{m^n}\bfu^{n-1}\in L^{2}$. Then we have
    \begin{align}
        \frac{1}{2} \Vert \sqrt{m^n}\bfu^{n}\Vert^{2}_{L^{2}} +
        k \left(\frac{c_o}{3}\Vert\bfu^n-\bfu^n_o\Vert^3_{L^3} + \D[P^n](\bfu^n)\right)
        \leq \frac{1}{2} \Vert \sqrt{m^n}\bfu^{n-1}\Vert^{2}_{L^{2}} + k g^n \ , \label{eq:DES}
    \end{align}
    where $g^n := \frac{4c_o}{3} \Vert \bfu^n_o\Vert^3_{L^3} + \frac{2}{3\sqrt{c_o}} \Vert \bff^n\Vert_{L^{3/2}}^{3/2}
    +\Vert\bff^n\Vert_{L^{3/2}}\Vert \bfu^n_o\Vert_{L^3}$. If the data $m^n \equiv m$, $P^n\equiv P$ are step-independent, summation over $n=1,\dots,N$ telescopes to
    \begin{align*}
        \frac{1}{2} \Vert \sqrt{m}\bfu^{N}\Vert^{2}_{L^{2}} +
        \sum_{n=1}^N k \left(\frac{c_o}{3}\Vert\bfu^n-\bfu^n_o\Vert^3_{L^3} + \D(\bfu^n)\right)
        \leq \frac{1}{2} \Vert \sqrt{m}\bfu^{0}\Vert^{2}_{L^{2}} + \sum_{n=1}^N k g^n \ .
    \end{align*}
\end{lemma}
\begin{proof}
    Testing \eqref{eq:vi_td} with $\bfv =0$ gives $\langle \mathcal{A}(\bfu^{n}),\bfu^n\rangle + \D[P^n](\bfu^n)-\langle \bff^n, \bfu^n\rangle \leq 0$. Similar to the coercivity proof of Theorem \ref{th:td}, we estimate with
    Young's inequality
    \begin{align*}
        &\langle \mathcal{A} (\bfu^{n}) - \bff^n , \bfu^n \rangle\\
        &\quad= \int_\Omega \left(\frac {m^n} k (\bfu^n-\bfu^{n-1})\cdot \bfu^n + c_o |d| d\cdot(d+\bfu^n_o)-\bff^n \cdot \bfu^n\right) \dd x\\
        &\quad\geq \int_\Omega \left(\frac{m^n}{2k} |\bfu^n|^{2} - \frac{m^n}{2k}|\bfu^{n-1}|^{2} + c_o |d|^3 - c_o |d|^2|\bfu^n_o|
        -|\bff^n|(|d|+|\bfu^n_o|)\right) \dd x\\
        &\quad\geq \int_\Omega \left(\frac{m^n}{2k} |\bfu^n|^{2} - \frac{m^n}{2k}|\bfu^{n-1}|^{2}\right) \dd x
        +\frac{c_o}{3} \Vert d \Vert^3_{L^3}\\
        &\qquad - \frac{4c_o}{3} \Vert \bfu^n_o\Vert^3_{L^3} - \frac{2}{3\sqrt{c_o}} \Vert \bff^n\Vert_{L^{3/2}}^{3/2}
        -\Vert\bff^n\Vert_{L^{3/2}}\Vert \bfu^n_o\Vert_{L^3} \ ,
    \end{align*}
    where $d := \bfu^n -\bfu^n_o$ and the same Young inequalities as in Theorem \ref{th:td}, Step 4, are used. Multiplication with $k$ gives \eqref{eq:DES}. For step-independent $m$ and $P$ the summation over $n$ telescopes the mass term.
\end{proof}

\section{Main result}
We define the space-time cylinder $Q_T = (0,T) \times \Omega$.
The operators $\div$ and $\varepsilon$ act on the spatial variable. We set
\begin{align}
    X:= \{\bfv\in W^{1,3}(\Omega)^2:\bfv_{|\Gamma_D}=0\} \ . \label{eq:def_X}
\end{align}
We assume that the time-dependent data satisfy
\begin{align}
    \bff \in L^{3/2}(0,T; L^{3/2}(\Omega)^2), \quad \bfu_o \in L^{3}(0,T; L^3(\Omega)^2), \quad
    \sqrt{m(\cdot,0)}\bfu_0 \in L^2(\Omega)^2   \ , \label{eq:data_time}
\end{align}
and that the ice mass and strength satisfy
\begin{align} \label{eq:time_m_P}
    m \in L^{\infty}(Q_T), \quad m\geq 0, \quad \partial_t m \in L^3(Q_T), \quad P \in C(\overline{Q_T}), \quad P \geq 0,
\end{align}
For some non-negative $\lambda\in L^1(0,T)$, we impose the growth condition
\begin{align}\label{eq:wachstumschranke}
    (\partial_t m)^+\leq\lambda(t)m\quad\text{a.e.\ in }Q_T,\qquad
    \beta(t):=\int_0^t\lambda(s)\dd s\ .
\end{align}
From \eqref{eq:time_m_P} it follows $m\in W^{1,3}(0,T;L^3(\Omega)) \hookrightarrow C([0,T];L^3(\Omega))$
by the Sobolev--Bochner embedding \cite[Section 7.1]{R13}.
Thus $m(\cdot,t)$ is well defined for every $t\in[0,T]$, in particular
at $t=0$ in \eqref{eq:data_time}. By \eqref{eq:wachstumschranke} and
Gr\"onwall's inequality, for a.e.\ $x\in\Omega$ we have
\begin{align}
    m(x,t) \leq \exp(\beta(t)-\beta(s))\, m(x,s), \quad 0\leq s \leq t \leq T \ , \label{eq:Gronwall}
\end{align}
in particular $\Vert m(\cdot,t)\Vert_{L^\infty(\Omega)} \leq \exp(\Vert\lambda\Vert_{L^1})\Vert m\Vert_{L^\infty(Q_T)}$ for every $t$.
Therefore the ice-free set is invariant in time: the bound allows melting everywhere, but
ice growth only where there is already ice.

We fix an admissible rate $\lambda$ and set
\begin{align*}
    \rho(t):=\exp(-\beta(t)),\qquad
    a(x,t):=\rho(t)m(x,t),\qquad
    \hat P(x,t):=\rho(t)P(x,t) \ .
\end{align*}
Then $\hat P\in C(\overline{Q_T})$ and
$a\in W^{1,1}(0,T;L^3(\Omega))\cap L^\infty(Q_T)$, with
\begin{align}\label{eq:wu-a-monotone}
    \partial_t a=\rho(\partial_t m-\lambda m)\leq0,\qquad
    a(\cdot,0)=m(\cdot,0)\ .
\end{align}
For fields with finite weighted norm, we use the seminorm
\begin{align*}
    \Vert\bfv\Vert_{a(t)}^2:=\int_\Omega a(x,t)|\bfv(x)|^2\dd x\ .
\end{align*}
For $\bfu\in L^3(0,T;V)$ and $\bfv\in C^1([0,T];X)$, define
\begin{align}\label{eq:wu-H}
    \Hh_{\bfu}(\bfv):={}&\int_0^T\rho(t)
        \langle m\partial_t\bfv,\bfv-\bfu\rangle\dd t
        +\frac{1}{2}\Vert\bfv(0)-\bfu_0\Vert_{a(0)}^2\ .
\end{align}
This functional collects the weighted time term and the initial term of
the weak formulation. Both terms are finite by the assumptions on the data.

We denote $\D_t:=\D[P(\cdot,t)]$ and $\Sigma_t:=\Sigma[P(\cdot,t)]$.
For the spatial operator, we set
\begin{align}
    \B(t;\bfv):=m(\cdot,t)f_c\bfe_3\times\bfv
       +c_o|\bfv-\bfu_o(t)|(\bfv-\bfu_o(t))
       \quad\text{for all }\bfv\in V\ .
\end{align}
Then $\B(t;\bfv)\in V^*$ for a.e.\ $t$. From Lemma \ref{lem:drag} we conclude that $\B(t;\cdot)$ is monotone and hemicontinuous with
\begin{align}
    \begin{split}
    \langle\B(t;\bfv_1)-\B(t;\bfv_2),\bfv_1-\bfv_2\rangle
    &\geq\frac{c_o}{2}\int_\Omega
       (|\bfv_1-\bfu_o(t)|+|\bfv_2-\bfu_o(t)|)|\bfv_1-\bfv_2|^2\dd x\\
    &\geq\frac{c_o}{2}\Vert\bfv_1-\bfv_2\Vert_{L^3}^3\ .
    \end{split}\label{eq:B_motone}
\end{align}
The operator $\bfv\mapsto\B(\cdot;\bfv)$ is continuous from $L^3(Q_T)^2$ to
$L^{3/2}(Q_T)^2$. The Coriolis term is bounded and linear,
and the drag satisfies
\begin{align*}
 \big||\boldsymbol{p}|\boldsymbol{p}-|\boldsymbol{r}|\boldsymbol{r}\big|
 \leq(|\boldsymbol{p}|+|\boldsymbol{r}|)|\boldsymbol{p}-\boldsymbol{r}|,
\end{align*}
which gives the required continuity by H\"older's inequality. This is the
standard continuity argument for Nemytskii operators \cite[Section 1.3]{R13}.

We consider the evolution problem
\begin{align}\label{eq:time_evolution}
    \bff(t) - m(\cdot,t)\partial_t\bfu(t)-\B(t;\bfu(t)) \in \partial\D_t(\bfu(t)) , \quad \bfu(0) = \bfu_0 \ .
\end{align}
Since $m$ is allowed to degenerate, time regularity of $\bfu$ cannot be expected.
The definition of a weak solution therefore shifts the time derivative
to the test function. An exponential weight accounts for the time dependence of $m$ through the growth condition \eqref{eq:wachstumschranke}.
The time dependence of $P$ is included by a space-time dissipation functional.
The weighted time term is given by \eqref{eq:wu-H}. We establish the required properties of the dissipation in the following two lemmas.

\begin{lemma}\label{lem:prop_time_D}
    Let $q_1,q_2 \in C(\bar \Omega)$ and $q\in C(\overline{Q_T})$ be non-negative.
    \begin{enumerate}
        \item[(a)] From $q_1\leq q_2$ it follows $\K(q_1(x)) \subseteq \K(q_2(x))$ for
        all $x$, hence $\Sigma[q_1]\subseteq\Sigma[q_2]$ and $\D[q_1]\leq \D[q_2]$
        on $V$. For non-negative constants $c$ we have $\D[cq_1] = c \D[q_1]$.
        \item[(b)] For $\bfv\in X$ we have
        \begin{align}
            \begin{split}
            \D[q_1](\bfv) - \D[q_2](\bfv) &= \int_\Omega (q_1-q_2)D(1,\varepsilon (\bfv)) \dd x,\\
            D(1,\varepsilon(\bfv)) &= \frac{1}{2}(\Delta(\varepsilon(\bfv))-\div\bfv),
            \end{split}
            \label{eq:time_D}
        \end{align}
        in particular $|\D[q_1](\bfv)-\D[q_2](\bfv)|\leq \Vert q_1-q_2\Vert_{C(\bar\Omega)}\int_\Omega D(1,\varepsilon(\bfv)) \dd x$ with $\int_\Omega
        D(1,\varepsilon(\bfv))\dd x \leq c_\K \Vert\varepsilon(\bfv)\Vert_{L^1}<\infty$.
        \item[(c)] For $\bfv \in C([0,T];X)$ the map $t\mapsto \D[q(\cdot,t)](\bfv(t))$ is continuous with
        \begin{align*}
            &|\D[q(\cdot, t)](\bfv(t))-\D[q(\cdot,s)](\bfv(s))|\\
            &\quad\leq\Vert q(\cdot,t)-q(\cdot,s)\Vert_{C(\bar\Omega)}\int_\Omega D(1,\varepsilon(\bfv(t))) \dd x \\
            &\qquad + c_\K \Vert q\Vert_{L^\infty(Q_T)}
            \Vert \varepsilon(\bfv(t)-\bfv(s))\Vert_{L^1} \ .
        \end{align*}
    \end{enumerate}
\end{lemma}
\begin{proof}
    \textit{(a)} The inclusion of $\K$ follows as in Lemma \ref{lema:prop_K}.
    The pointwise inclusion is inherited by $\Sigma$, and for $\D$ the supremum
    over the smaller set is smaller. For the scaling, $\Sigma[cq_1]=c\Sigma[q_1]$ and
    the tested functional is linear in $\bftau$.

    \textit{(b)} The identity \eqref{eq:int_darstellung} holds for any $q\in C(\bar\Omega)$, $q\geq 0$, since the proof only uses these properties. The linearity of
    the support function in the strength and the identity \eqref{eq:def_D} lead to \eqref{eq:time_D}.
    The estimate follows by $0\leq D(1,\bfxi) \leq c_\K |\bfxi|_F$ (Lemma \ref{lem:prop_D}) and $\varepsilon(\bfv)\in L^3\hookrightarrow L^1$.

    \textit{(c)} Use the triangle inequality and estimate the first term by (b).
    The second term follows by the representation and the Lipschitz estimate of Lemma \ref{lem:prop_D}. Both terms vanish for $s\rightarrow t$ and show the continuity.
\end{proof}

Similar to Definition \ref{def:D} we define the space-time dissipation functional
for a $q\in C(\overline{Q_T})$, $q\geq 0$, as follows:
\begin{align*}
    \begin{split}
    \Sigma_{Q_T}[q]:= \{ \bftau \in L^\infty(Q_T;\mathbb S) : {}&\bftau(\cdot,t)\in
    \Sigma[q(\cdot,t)] \text{ for a.e.\ } t,\\
    &\text{there exists } \bfw\in L^{3/2}(Q_T)^2:\\
    &\bfw(\cdot,t) = \div \bftau(\cdot,t) \text{ for a.e.\ } t \} \ .
    \end{split}
\end{align*}
We denote
\begin{align*}
    \D_{Q_T}[q](\bfw) &:= \sup_{\bftau\in\Sigma_{Q_T}[q]}\Big(-\int_{0}^{T}\int_{\Omega} \bfw\cdot \div\bftau \dd x \dd t\Big),\\
    \D_{Q_T}&:=\D_{Q_T}[P]
\end{align*}
for $\bfw\in L^3(Q_T)^2$.
We use the same definitions for non-negative coefficients $q$ that are
piecewise constant in time on a finite partition, with values in
$C(\bar\Omega)$. Their values at the partition nodes do not affect the
admissible stress set or the dissipation functional.

\begin{lemma}\label{lem:DQT}
    Let $q\geq0$ be either continuous on $\overline{Q_T}$ or piecewise constant
    in time as above.
    \begin{enumerate}
        \item[(a)] The functional $\D_{Q_T}[q] : L^3(Q_T)^2\rightarrow [0,\infty]$ is  convex, positively homogeneous of degree 1 and weakly lower semicontinuous
    with respect to weak convergence in $L^3(Q_T)^2$. For coefficients $q_1,q_2$
    in these classes, $q_1\leq q_2$ a.e.\ on $Q_T$ implies
    $\D_{Q_T}[q_1]\leq\D_{Q_T}[q_2]$. For constant $c\geq 0$ we have $\D_{Q_T}[cq] = c \D_{Q_T}[q]$.
    \item[(b)] If $\bar q$ and $\bar \bfw$ are piecewise constant in $t$ with respect to
    a partition $\{I_n\}_1^N$ of $(0,T)$, $\bar q_{|I_n} = q_n \in C(\bar \Omega)$
    with $q_n\geq 0$, $\bar\bfw_{|I_n} = \bfw_n \in V$, we have
    \begin{align*}
        \D_{Q_T}[\bar q](\bar \bfw) \leq \sum_{n=1}^N |I_n| \D[q_n](\bfw_n) \ .
    \end{align*}
    \item[(c)] If $q\in C(\overline{Q_T})$, then for $\bfv \in C([0,T]; X)$ we have
    \begin{align*}
        \D_{Q_T}[q](\bfv) = \int_{0}^{T} \D[q(\cdot,t)] (\bfv(t)) \dd t \ ,
    \end{align*}
        where the integrand is continuous (Lemma \ref{lem:prop_time_D}(c)).
    \end{enumerate}
\end{lemma}
\begin{proof}
    \textit{(a)} As in Lemma \ref{lem:bdd}, $\D_{Q_T}[q]$ is the pointwise supremum of the linear, weakly $L^{3}$-continuous
    functionals $\bfw \mapsto - \int\int \bfw \cdot\div\bftau \dd x \dd t$, hence convex and weakly lower semicontinuous.
    Since $0 \in \Sigma_{Q_T}[q]$, we have $\D_{Q_T}[q]\geq 0 = \D_{Q_T}[q](0)$, and the supremum scales.

    Lemma \ref{lem:prop_time_D}(a) shows that $\Sigma[q_1(\cdot,t)] \subseteq \Sigma[q_2(\cdot,t)]$ for a.e.\ $t$,
    so $\Sigma_{Q_T}[q_1] \subseteq \Sigma_{Q_T}[q_2]$. Scaling with constants follows from $\Sigma_{Q_T}[cq] =c\Sigma_{Q_T}[q]$.

    \textit{(b)} Let $\bftau \in \Sigma_{Q_T}[\bar q]$. For a.e.\ $t\in I_n$ we have
    $\bftau(\cdot,t) \in \Sigma[q_n]$. Definition \ref{def:D} gives
    \begin{align*}
        \langle -\div\bftau(\cdot, t), \bfw_n\rangle \leq \D [q_n](\bfw_n) \ .
    \end{align*}
    We integrate over $I_n$, sum up and take the supremum over $\bftau$.

    \textit{(c)}  $"\leq"$: For $\bftau \in \Sigma_{Q_T}[q]$ we have for a.e.\ $t$ pointwise $\langle-\div \bftau(\cdot,t),
    \bfv(t)\rangle \leq \D[q(\cdot,t)](\bfv(t))$. Integration and taking the supremum lead to the estimate.

    $"\geq"$: Let $\eta >0$, $\omega_q$ the modulus of continuity of $q$ on $\overline{Q_T}$, $\omega_{\varepsilon(\bfv)}$
    the modulus of continuity of $t\mapsto \varepsilon(\bfv(t))$ in $L^{1}$, and $C_v := \sup_t \int_\Omega D(1,\varepsilon(\bfv(t)))\dd x<\infty$.
    Choose an equidistant partition $\{I_j\}_1^{J}$ of $(0,T)$ with mesh size $\kappa$ and nodes $t_j \in \bar I_j$,
    and set $\delta := \omega_q(\kappa)$. From the definition of the supremum in \eqref{eq:def_diss} there
    exists a $\bftau_j \in \Sigma[(q(\cdot,t_j)-\delta)^{+}]$ with
    \begin{align*}
        \langle - \div \bftau_j, \bfv(t_j)\rangle \geq \D[(q(\cdot,t_j)-\delta)^+](\bfv(t_j)) - \eta  \ .
    \end{align*}
    For $t\in I_j$ we have $q(\cdot,t)\geq q(\cdot,t_j) - \delta$ and $q(\cdot, t) \geq 0$, so $(q(\cdot,t_j)-\delta)^+ \leq q(\cdot,t)$
    and therefore $\bftau_j\in \Sigma[q(\cdot,t)]$ (Lemma \ref{lem:prop_time_D} (a)). The composite function
    $\bftau(\cdot,t) := \bftau_j$ for $t\in I_j$ therefore lies in $\Sigma_{Q_T}[q]$. For $t\in I_j$
    we apply the Gauss theorem and $|\bftau_j|\leq c_\K \Vert q\Vert_{L^{\infty}(Q_T)}$ a.e.:
    \begin{align*}
        \langle - \div \bftau_j, \bfv(t) \rangle &= \int_\Omega \bftau_j : \varepsilon(\bfv(t)) \dd x\\
        &\geq \int_\Omega \bftau_j : \varepsilon(\bfv(t_j)) \dd x - c_\K \Vert q\Vert_{L^{\infty}(Q_T)} \omega_{\varepsilon(\bfv)}(\kappa)\\
        &\geq \D[(q(\cdot,t_j)-\delta)^{+}](\bfv(t_j)) - \eta -  c_\K \Vert q\Vert_{L^{\infty}(Q_T)} \omega_{\varepsilon(\bfv)}(\kappa)\\
        &\geq \D[q(\cdot,t_j)](\bfv(t_j)) - \delta C_v -\eta -  c_\K \Vert q\Vert_{L^{\infty}(Q_T)} \omega_{\varepsilon(\bfv)}(\kappa)\\
        &\geq \D[q(\cdot,t)](\bfv(t)) - 2\omega_q(\kappa) C_v -\eta - 2 c_\K \Vert q\Vert_{L^{\infty}(Q_T)} \omega_{\varepsilon(\bfv)}(\kappa)
    \end{align*}
    where we applied Lemma \ref{lem:prop_time_D}(b) in the third line and Lemma \ref{lem:prop_time_D}(c) in the fourth line. Integration over $(0,T)$ yields
    \begin{align*}
        \D_{Q_T}[q](\bfv) &\geq \int_{0}^{T} \D[q(\cdot,t)](\bfv(t)) \dd t\\
        &\quad - T ( 2\omega_q(\kappa) C_v+\eta + 2 c_\K \Vert q\Vert_{L^{\infty}(Q_T)} \omega_{\varepsilon(\bfv)}(\kappa)) \ .
    \end{align*}
    We conclude the proof by taking the limits $\kappa \rightarrow 0$ and $\eta \rightarrow 0$.
\end{proof}

\begin{definition}\label{def:solution_evolution}
    A function $\bfu \in L^3(0,T;V)$ is a weak solution of \eqref{eq:time_evolution} if
    $\sqrt{m}\bfu \in L^\infty(0,T;L^2)$, $\D_{Q_T}(\bfu)<\infty$ and
    for all $\bfv \in C^1([0,T];X)$ we have
\begin{align}
    \begin{split}\label{eq:weak_solution}
        \Hh_{\bfu}(\bfv)&+\int_0^T\rho(t)
        \left[\langle\B(t;\bfv)-\bff,\bfv-\bfu\rangle
        +\D_t(\bfv(t))\right]\dd t\\
        &-\D_{Q_T}[\hat P](\bfu)\geq0\ .
    \end{split}\end{align}

\end{definition}

\begin{remark}
For sufficiently regular functions, set $\bfw:=\bfv-\bfu$. Then
\begin{align*}
    \int_{0}^{T}\rho \langle m \partial_t \bfu,\bfw\rangle \dd t
    = \int_{0}^{T} \rho \langle m \partial_t \bfv , \bfw\rangle \dd t +
    \frac 1 2 \Vert \sqrt{m(\cdot,0)} \bfw(0)\Vert_{L^{2}}^2 \\
    -\frac{\rho(T)}{2} \Vert \sqrt{m(\cdot, T)} \bfw (T)\Vert^{2}_{L^2}
    + \frac 1 2 \int_0^T \rho\int_\Omega (\partial_t m-\lambda m)|\bfw|^2 \dd x \dd t
\end{align*}
where the last two terms are non-positive, the first trivially and the second by \eqref{eq:wachstumschranke}.
If we test \eqref{eq:time_evolution} with $\bfw(t)$, multiply by $\rho(t)>0$, integrate and
shift $\B$ by monotonicity to the test function, the two non-positive terms can be discarded.
Thus every sufficiently regular solution of \eqref{eq:time_evolution} is a weak solution, for every
admissible rate $\lambda$. The definition depends on the chosen admissible rate.
For time-independent $P$, Lemma \ref{lem:wu-unshift} below recovers a
variational inequality with $\B$ evaluated at the solution. For time-independent $m$
the canonical choice is $\lambda \equiv 0$, with $\beta \equiv 0$, $\hat P =P$, and the weak solution
becomes the classical unweighted form. To see that the terms are well defined, note that $t\mapsto \rho(t)\D_t(\bfv(t))$
is continuous (Lemma \ref{lem:prop_time_D}(c), $\beta$ is absolutely continuous), therefore we have a classical time integral.
On the solution side, $\D_{Q_T}[\hat P](\bfu)$ avoids requiring a separate
measurability and integral-representation result for
$\int_0^T\rho\D_t(\bfu)\dd t$.
For $\bfu\in C([0,T];X)$ they coincide (Lemma \ref{lem:DQT}(c)), and by Lemma \ref{lem:DQT}(a) the
finiteness of $\D_{Q_T}[\hat P](\bfu)$ is equivalent to that of $\D_{Q_T}(\bfu)$, since $\exp(-\Vert\lambda\Vert_{L^1})P\leq \hat P\leq P$.
\end{remark}

\subsection{Main result}
We use the following discretisation: $t_n := nk$, $k=T/N$, $I_n := (t_{n-1},t_n]$, averaged data
$\bff^n := \frac{1}{k} \int_{I_n} \bff \dd t$, $\bfu^n_o := \frac 1 k \int_{I_n} \bfu_o \dd t$.
The mass, strength and weights are discretised as
\begin{align*}
    m^n &:= m(\cdot, t_n), \quad P^n:=\exp(\beta(t_n)) \max_{t\in \bar I_n} \hat P(\cdot, t),\\
    \rho_n &:= \rho(t_n)\in [\exp(-\Vert\lambda\Vert_{L^1}),1] \ .
\end{align*}
$\bfu^n$ is the solution of the time-discrete problem \eqref{eq:sub_vi_td} (Theorem \ref{th:td}) with
the data $(m^n, P^n,\bff^n,\bfu^n_o)$ and the initial value $\bfu^{n-1}$, $\bfu^0 := \bfu_0$.
$\bar\bfu_k$ is the piecewise constant interpolation ($\bar\bfu_{k|I_n} = \bfu^n$) and $\hat\bfu_k$
the affine interpolation. Analogously we have $\bar\bff_k, \bar\bfu_{o,k}, \bar m_k, \bar \rho_k$.
The choices are well defined and admissible:
\begin{enumerate}
    \item[(A)] $m^n$ is well defined since $m \in C([0,T];L^3)$, and \eqref{eq:Gronwall} leads to
    $m^{n} \in L^\infty(\Omega)$, $m^n\geq 0$ and pointwise a.e.
    \begin{align}
        \rho_n m^{n} \leq \rho_{n-1}m^{n-1} \quad\text{for }n\geq2, \quad \rho_1m^{1}\leq  m(\cdot,0),
        \label{eq:ass_m}
    \end{align}
    in particular $\sqrt{m^{1}}\bfu_0 \in L^{2}(\Omega)$. Therefore the data satisfy
    Assumption \ref{ass} and $\bfu^{n}$ exists uniquely (Theorem \ref{th:td}).
    \item[(B)] $P^{n} \in C(\bar\Omega)$, $P^{n}\geq 0$, and for $t\in I_n$ we have pointwise
    \begin{align}   \label{eq:B}
        \begin{split}
        &\rho_n P^{n} = \max_{\bar I_n}\hat P \geq \hat P(\cdot,t), \quad P^{n}\geq P(\cdot,t), \\
        &\Vert P^{n}-P(\cdot, t)\Vert_{C(\bar\Omega)} \leq \exp(\Vert\lambda\Vert_{L^{1}}) \omega_{\hat P}(k)
        +(\exp(\mu_k) -1) \Vert P\Vert_{L^{\infty}(Q_T)}
        \end{split}
    \end{align}
    with the modulus of continuity $\omega_{\hat P}$ of $\hat P$ on $\overline{Q_T}$ and $\mu_k :=
    \max_n \int_{I_n} \lambda \dd s \rightarrow 0$ for $k\rightarrow 0$. The right side is independent
    of $n$ and vanishes for $k\rightarrow 0$. Also $\Vert\bar\rho_k-\rho\Vert_{L^{\infty}(0,T)} \leq \mu_k \rightarrow 0$.
\end{enumerate}
\begin{theorem}\label{th:rothe}
    Under Assumption \ref{ass} (A1), $f_c\in L^{\infty}(\Omega)$, $c_o >0$, \eqref{eq:data_time},
    \eqref{eq:time_m_P} and \eqref{eq:wachstumschranke} the following statements hold:
    \begin{enumerate}
        \item[(a)] A subsequence of $(\bar\bfu_k)_k$ converges weakly in $L^{3}(Q_T)$ to a
        weak solution $\bfu$ in the sense of Definition \ref{def:solution_evolution}, and every such limit satisfies
        the energy estimate
        \begin{align*}
            \esssup_{t\in (0,T)} \frac{1}{2}\Vert \sqrt{m(\cdot,t)}\bfu(t)\Vert^{2}_{L^{2}}
            +\int_{0}^{T} \frac{c_o}{3} \Vert \bfu-\bfu_o\Vert^{3}_{L^{3}} \dd t
            +\D_{Q_T}(\bfu) \\
            \leq 2\exp(\Vert\lambda\Vert_{L^{1}}) \left(\frac{1}{2}\Vert\sqrt{m(\cdot,0)}\bfu_0\Vert^{2}_{L^{2}}
            +\int_{0}^{T}g(t)\dd t\right)
        \end{align*}
        with $g(t):= \frac{4c_o}{3} \Vert \bfu_o(t)\Vert_{L^{3}}^{3} + \frac{2}{3\sqrt{c_o}}\Vert \bff(t)\Vert_{L^{3/2}}^{3/2} + \Vert\bff(t)\Vert_{L^{3/2}}\Vert\bfu_o(t)\Vert_{L^{3}}$.
        \item[(b)] If $P(x,t)=P(x)$ is independent of time, there exists
        exactly one weak solution in the sense of
        Definition \ref{def:solution_evolution}, for the fixed admissible
        rate $\lambda$. Uniqueness concerns the velocity a.e.\ on $Q_T$,
        including the sets where $m$ or $P$ vanishes.
    \end{enumerate}

\end{theorem}
\subsection{Proof of (a)}
\begin{proof}
    We use the time-discretisation strategy of Rothe's method
    \cite[Section 8.2]{R13}, with weighted estimates for the present coefficients.

    Step 1:

    Jensen's inequality leads to $\sum_n k \Vert \bff^{n}\Vert_{L^{3/2}}^{3/2} \leq \Vert\bff\Vert_{L^{3/2}}^{3/2}$
    and analogously for $\bfu_o$. Applying H\"older to the mixed term leads to $\sum_n k g^{n}\leq C(\bff, \bfu_o)$.
    The discrete energy estimate \eqref{eq:DES} for the data $(m^{n},P^{n},\bff^{n},\bfu^{n}_o)$ is multiplied
    by $\rho_n \leq 1$ and summed up. Due to \eqref{eq:ass_m} we have $\rho_n \Vert \sqrt{m^{n}}\bfu^{n-1}\Vert^{2}_{L^{2}}
    \leq \rho_{n-1}\Vert\sqrt{m^{n-1}}\bfu^{n-1}\Vert^{2}_{L^{2}}$ for $n\geq 2$ and $\rho_1 \Vert \sqrt{m^{1}}\bfu^{0}
    \Vert^{2}\leq \Vert \sqrt{m(\cdot,0)}\bfu_0\Vert^{2}$. The weighted sum telescopes, and for any
    $N'\leq N$ it follows
    \begin{align}
        \begin{split}\label{eq:weighted_discrete_energie}
        \frac{\rho_{N'}}{2}\Vert\sqrt{m^{N'}}\bfu^{N'}\Vert^{2}_{L^{2}} &+ \sum_{n=1}^{N'}k \rho_n
        \left(\frac{c_o}{3}\Vert\bfu^{n}-\bfu^{n}_o\Vert^{3}_{L^{3}}+\D[P^{n}](\bfu^{n})\right)\\
        &\leq \frac{1}{2} \Vert \sqrt{m(\cdot,0)}\bfu_0\Vert^{2}_{L^{2}} + \sum_{n=1}^{N}k g^{n} \leq C \ .
                \end{split}
    \end{align}
    With $\rho_n \geq \exp(-\Vert\lambda\Vert_{L^{1}})$ and $\sum_n k\Vert\bfu^{n}_o\Vert_{L^{3}}^{3}
    \leq \Vert\bfu_o\Vert^{3}_{L^{3}}$, in particular $(\bar\bfu_k)$ is bounded in $L^{3}(Q_T)$,
    $(\sqrt{\bar m_k}\bar\bfu_k)$ bounded in $L^{\infty}(0,T;L^{2})$ and $\sum_n k \D[P^{n}](\bfu^{n})\leq C$.

    Step 2:

    We take a subsequence, with the same notation, $\bar\bfu_k \rightharpoonup \bfu$ in $L^{3}(Q_T)$. The averaged data converge strongly, $\bar \bff_k \rightarrow \bff$ in $L^{3/2}(Q_T)$, $\bar\bfu_{o,k}\rightarrow\bfu_o$ in
    $L^{3}(Q_T)$. From the uniform continuity of $m: [0,T] \rightarrow L^{3}(\Omega)$ it follows
    \begin{align*}
        \Vert\bar m_k-m\Vert_{L^{\infty}(0,T;L^{3})} \leq
    \sup_{|t-s|\leq k} \Vert m(t)-m(s)\Vert_{L^{3}} \rightarrow 0
    \end{align*}
    and with $|\sqrt{a}-\sqrt{b}|\leq \sqrt{|a-b|}$ therefore
    \begin{align*}
        \Vert \sqrt{\bar m_k}-\sqrt{m}\Vert_{L^{\infty}(0,T;L^{6})} \leq \Vert\bar m_k-m\Vert^{1/2}_{L^{\infty}(0,T;L^{3})}\rightarrow 0 \ .
    \end{align*}
    Assumption (B) shows that $\bar\rho_k \rightarrow \rho$ uniformly.
    From the boundedness of Step 1 and identification of the limit via test functions it follows that
     $\sqrt{\bar m_k}\bar \bfu_k \wstar \sqrt{m}\bfu$ in $L^{\infty}(0,T;L^{2})$,
     in particular $\sqrt{m}\bfu \in L^{\infty}(0,T;L^{2})$.

     Step 3: We use monotonicity to evaluate the nonlinear operator at the
     test function, as in the Minty formulation \cite[Section 2.3]{R13}.

    Let $\bfv \in C^1([0,T];X)$, $\bfv^{n}:= \bfv(t_n)$. Testing the discrete VI \eqref{eq:vi_td} in time step $n$ with $\bfv^{n}$,
    the monotonicity
    \begin{align*}
        &\langle \B^{n}(\bfu^{n})-\B^n(\bfv^n), \bfv^n-\bfu^n\rangle \leq 0 \\
        &\B^n(\bfw) := m^nf_c\bfe_3 \times \bfw + c_o |\bfw-\bfu^n_o|(\bfw-\bfu_o^n)
    \end{align*}
    and the fact that the Coriolis term can be shifted due to pointwise skew symmetry, independently of $m^n$, we have
    \begin{align}\label{eq:minty}
        \begin{split}
        \left\langle \frac{m^n}{k}(\bfu^n-\bfu^{n-1}),\bfv^n-\bfu^n \right\rangle
        + \langle\B^n(\bfv^n)-\bff^n,\bfv^n-\bfu^n\rangle \\+ \D[P^n](\bfv^n)-\D[P^n](\bfu^n)\geq 0 \ .
        \end{split}
    \end{align}

    Step 4:

    Set $\bfw^n := \bfv^n -\bfu^n$ and calculate
    \begin{align*}
        \langle m^n(\bfu^n-\bfu^{n-1}),\bfw^n\rangle = \langle m^n (\bfv^n-\bfv^{n-1}),\bfw^n\rangle - \langle m^n(\bfw^n-\bfw^{n-1}),\bfw^n\rangle \ .
    \end{align*}
    The estimate $\langle m^n (a-b),a\rangle \geq \frac 1 2 (\Vert\sqrt{m^n}a\Vert^2-\Vert\sqrt{m^n}b\Vert^2)$, the rearrangement
    \begin{align*}
        \sum_{n=1}^{N} \rho_n (a_n-b_n) = \rho_Na_N-\rho_1b_1 + \sum_{n=2}^{N}(\rho_{n-1}a_{n-1}-\rho_nb_n)
    \end{align*}
    for $a_n := \Vert \sqrt{m^n}\bfw^n\Vert^2_{L^2}$ and $b_n:=\Vert\sqrt{m^n}\bfw^{n-1}\Vert^2_{L^2}$, \eqref{eq:ass_m} and
    \begin{align*}
        \rho_{n-1}a_{n-1}-\rho_nb_n &= \int_\Omega (\rho_{n-1}m^{n-1}-\rho_nm^n)
        |\bfw^{n-1}|^2\dd x \geq 0,\\
        \rho_1b_1 &\leq \Vert\sqrt{m(\cdot,0)}\bfw^0\Vert^2_{L^2}\ ,
    \end{align*}
    lead, for any $k$, to
    \begin{align*}
        &\sum^N_{n=1} \rho_n \langle m^n (\bfu^n-\bfu^{n-1}),\bfw^n\rangle\\
        &\quad\leq \sum_{n=1}^{N}\rho_n\langle m^n(\bfv^n-\bfv^{n-1}),\bfw^n\rangle
        +\frac 1 2 \Vert\sqrt{m(\cdot,0)}(\bfv(0)-\bfu_0)\Vert^2_{L^2} \ .
    \end{align*}
    The two terms $\rho_N a_N \geq 0$ and the sum possess the right sign thanks
    to the growth rate condition \eqref{eq:wachstumschranke}. This summation-by-parts step uses the same mass monotonicity as the energy estimate. For the first term
    on the right-hand side we have the difference quotient $\dd_k \bfv^n =
    (\bfv^n-\bfv^{n-1})/k = \frac{1}{k}\int_{I_n} \partial_t \bfv \dd s$.
    The step function $t\mapsto \rho_n m^n \dd_k \bfv^n$ for $t\in I_n$ converges
    strongly in $L^{3/2}(Q_T)$ to $\rho m\partial_t\bfv$, together with $\bar \bfv_k -\bar \bfu_k \rightharpoonup \bfv-\bfu$ in $L^3(Q_T)$. So
    \begin{align}
        \sum_{n=1}^{N}k\rho_n\langle m^n \dd_k\bfv^n, \bfv^n-\bfu^n
        \rangle \rightarrow \int_{0}^{T}\rho \langle m\partial_t \bfv,
        \bfv-\bfu\rangle \dd t \ . \label{eq:limit_time}
    \end{align}

    Step 5:
    We multiply \eqref{eq:minty} by $k\rho_n$ and sum up, then $k\rightarrow 0$
    along a subsequence:
    \begin{itemize}
        \item The step function $t\mapsto \rho_n (\B^n(\bfv^n)-\bff^n)$ converges
        strongly in $L^{3/2}(Q_T)$. For the Coriolis term, this follows by H\"older's inequality from $\bar m_k\rightarrow m$ in $L^\infty(0,T;L^3)$ and $\sup_{t\in[0,T]}\Vert\bar\bfv_k(t)-\bfv(t)\Vert_X\rightarrow0$. For the ocean drag, we use continuity of $\bfw\mapsto|\bfw|\bfw$ from $L^3(Q_T)$ to $L^{3/2}(Q_T)$, convergence of the averaged data as in Step 2 and uniform convergence of the weights. Together with $\bar \bfv_k-\bar \bfu_k \rightharpoonup \bfv-\bfu$ we conclude
        \begin{align*}
            \sum_n k\rho_n\langle\B^n(\bfv^n)-\bff^n,\bfv^n-\bfu^n\rangle
            \rightarrow \int_{0}^{T}\rho \langle \B(t;\bfv)-\bff,\bfv-\bfu\rangle \dd t \ .
        \end{align*}
        \item For $t\in I_n$ we have
        \begin{align*}
            |\rho_n\D[P^n](\bfv^n)-\rho(t)\D_t(\bfv(t))|
            \leq \mu_k\D[P^n](\bfv^n) + |\D[P^n](\bfv^n)-\D_t(\bfv(t))|
        \end{align*}
        and both terms vanish uniformly in $n$. The first one since
        \begin{align*}
            \D[P^n](\bfv^n) = \int_\Omega P^nD(1,\varepsilon(\bfv^n)) \dd x
            \leq c_\K \sup_n\Vert P^n\Vert_{L^\infty} \sup_t\Vert\varepsilon(\bfv(t))\Vert_{L^1} \leq C
        \end{align*}
        and the second one with Lemma \ref{lem:prop_time_D}(b),(c) and Assumption (B). Therefore
        \begin{align*}
            \sum_n k \rho_n \D[P^n](\bfv^n) \rightarrow \int_0^T\rho(t)\D_t(\bfv(t))\dd t \ .
        \end{align*}
        \item For the scaling we have $\rho_n \D[P^n]=\D[\rho_nP^n]$ (Lemma \ref{lem:prop_time_D}(a)), and $\bar q_{k|I_n}:= \rho_n P^n = \max_{\bar I_n} \hat P$ satisfies $\bar q_k\geq \hat P$ in $Q_T$ (Assumption (B)).
        So with Lemma \ref{lem:DQT}(a),(b) and monotonicity we conclude
        \begin{align*}
            \sum_n k \D[\rho_nP^n](\bfu^n) &\geq \D_{Q_T}[\bar q_k](\bar \bfu_k)
            \geq \D_{Q_T}[\hat P](\bar\bfu_k),\\
            \liminf_k \D_{Q_T}[\hat P](\bar\bfu_k)&\geq \D_{Q_T}[\hat P](\bfu) \ ,
        \end{align*}
        and with this we have
        \begin{align*}
            \limsup_k \Big(-\sum_n k\rho_n\D[P^n](\bfu^n)\Big) \leq -\D_{Q_T}[\hat P](\bfu) \ .
        \end{align*}
        \item The time term is the limit \eqref{eq:limit_time} from Step 4.
    \end{itemize}
    Thus $\bfu$ satisfies the inequality \eqref{eq:weak_solution} and
    \begin{align*}
        \D_{Q_T}(\bfu) \leq \exp(\Vert\lambda\Vert_{L^1})\D_{Q_T}[\hat P](\bfu)
        \leq \exp(\Vert\lambda\Vert_{L^1})\liminf_k\sum_n k \rho_n\D[P^n](\bfu^n)
        <\infty
    \end{align*}
    by Lemma \ref{lem:DQT}(a) and \eqref{eq:weighted_discrete_energie}.
    Thus $\bfu$ is a solution in the sense of Definition \ref{def:solution_evolution}.

    For the energy estimate set
    \begin{align*}
        C_k := \exp(\Vert\lambda\Vert_{L^1})
        \left(\frac{1}{2}\Vert\sqrt{m(\cdot,0)}\bfu_0\Vert^2_{L^2}+\sum_{n=1}^{N}kg^n\right) \ .
    \end{align*}
    Multiplying \eqref{eq:weighted_discrete_energie} by $\exp(\Vert\lambda\Vert_{L^1})$
    and using $\exp(\Vert\lambda\Vert_{L^1})\rho_n\geq1$, we obtain for every $N'\leq N$
    \begin{align*}
        \frac{1}{2} \Vert\sqrt{m^{N'}}\bfu^{N'}\Vert^2_{L^2}
        +\sum_{n=1}^{N'}k\left(\frac{c_o}{3}\Vert\bfu^n-\bfu^n_o\Vert^3_{L^3} +\D[P^n](\bfu^n)\right) \leq C_k \ .
    \end{align*}
    All terms on the left are non-negative. Taking the maximum of the kinetic energy over $N'$
    gives a bound by $C_k$, while choosing $N'=N$ bounds the total dissipation by $C_k$.
    Adding these two bounds yields
    \begin{align*}
        \frac{1}{2}\Vert\sqrt{\bar m_k}\bar\bfu_k\Vert_{L^\infty(0,T;L^2)}^2
        +\int_0^T\frac{c_o}{3}\Vert\bar\bfu_k-\bar\bfu_{o,k}\Vert_{L^3}^3\dd t
        +\sum_{n=1}^{N}k\D[P^n](\bfu^n)\leq2C_k \ .
    \end{align*}
    By the weak-$\ast$ convergence established in Step 2 and lower semicontinuity of the norm,
    \begin{align*}
        \esssup_{t\in(0,T)}\frac{1}{2}\Vert\sqrt{m(\cdot,t)}\bfu(t)\Vert_{L^2}^2
        \leq\liminf_k\frac{1}{2}\Vert\sqrt{\bar m_k}\bar\bfu_k\Vert_{L^\infty(0,T;L^2)}^2 \ .
    \end{align*}
    The ocean drag term is weakly lower semicontinuous (convex integral functional and $\bar\bfu_{o,k}\rightarrow \bfu_o$ strongly in $L^3(Q_T)$). For
    the dissipation term we have $P^n\geq P(\cdot,t)$ in $I_n$ (Assumption (B)), so
    as in Step 5, $\sum_n k\D[P^n](\bfu^n)\geq\D_{Q_T}(\bar\bfu_k)$ and thus
    \begin{align*}
        \liminf_k\sum_n k\D[P^n](\bfu^n) \geq \D_{Q_T}(\bfu) \ .
    \end{align*}
    Finally,
    \begin{align*}
        \limsup_k\sum_n kg^n \leq \int_{0}^{T}g \dd t
    \end{align*}
    by Jensen's inequality. The mixed terms converge since $t \mapsto \Vert \bar\bff_k(t)\Vert_{L^{3/2}}$ converges in $L^{3/2}(0,T)$ and $t \mapsto \Vert\bar\bfu_{o,k}(t)\Vert_{L^3}$ in $L^3(0,T)$, hence the product in $L^1(0,T)$.
    Applying these lower semicontinuity estimates to the sum and using the bound on $\limsup_k C_k$
    proves the asserted energy estimate with the factor $2$.
\end{proof}
\subsection{Proof of (b)}
Throughout this subsection, $P(x,t)=P(x)$ and the same admissible rate
$\lambda$ is fixed for both solutions. We use the notation introduced at
the beginning of this section. Since $P$ is independent of time, we have
$\D_t=\D[P]=\D$ for all $t\in[0,T]$.
Lemma \ref{lem:DQT} gives convexity and lower semicontinuity on $L^3(Q_T)^2$ and,
for $\bfv\in C^1([0,T];X)$, the representation
\begin{align}\label{eq:wu-F-smooth}
 \D_{Q_T}[\hat P](\bfv)=\int_0^T\rho(t)\D(\bfv(t))\dd t
       =\int_{Q_T} D(\hat P,\varepsilon(\bfv))\dd x\dd t \ .
\end{align}
The spatial Green identity on $X$ follows from Lemma \ref{lem:a3} applied
to both signs of a field with zero Dirichlet trace.

By \eqref{eq:wu-F-smooth}, the dissipation of the test function in
\eqref{eq:weak_solution} equals $\D_{Q_T}[\hat P](\bfv)$. We use this
representation throughout the proof.

\subsubsection{Approximation with convergence of the dissipation.}

\begin{lemma}\label{lem:wu-relaxation}
For every $\bfw\in L^3(Q_T)^2$ with $\D_{Q_T}[\hat P](\bfw)<\infty$, there is
a sequence $(\bfw_j)_j\subset C^1([0,T];X)$ such that
\begin{align*}
 \bfw_j\rightarrow \bfw\quad\text{in }L^3(Q_T)^2,\qquad
 \D_{Q_T}[\hat P](\bfw_j)\rightarrow\D_{Q_T}[\hat P](\bfw) \ .
\end{align*}
\end{lemma}

\begin{proof}
Define the convex, positively homogeneous functional on $L^3(Q_T)^2$
\begin{align*}
     F(\bfv):=
 \begin{cases}
  \displaystyle\int_{Q_T} D(\hat P,\varepsilon(\bfv))\dd x\dd t,&\text{if }\bfv\in C^1([0,T];X),\\
  +\infty,&\text{if }\bfv\notin C^1([0,T];X) \ .
 \end{cases}
\end{align*}
We show that the lower semicontinuous relaxation of $F$ is exactly
$\D_{Q_T}[\hat P]$.
Its convex conjugate is the indicator function of
\begin{align*}
 C:=\{\bfxi\in L^{3/2}(Q_T)^2: \langle\bfxi,\bfv\rangle\leq F(\bfv)
                         \text{ for all }\bfv\in C^1([0,T];X)\} \ .
\end{align*}
Indeed, positive homogeneity and $F(0)=0$ give $F^*(\bfxi)=0$
for $\bfxi\in C$ and $F^*(\bfxi)=+\infty$ otherwise. We claim that
\begin{align}\label{eq:wu-polar}
 C=\{-\div\bftau:\bftau\in\Sigma_{Q_T}[\hat P]\} \ .
\end{align}
The first inclusion follows from the Green identity and
$\bftau:\varepsilon(\bfv)\leq D(\hat P,\varepsilon(\bfv))$.
For the reverse inclusion, choose $\bfxi\in C$. We define a linear functional
on the subspace $\varepsilon(C^1([0,T];X))$ of $L^1(Q_T;\mathbb S)$ by
\begin{align*}
    \ell(\varepsilon(\bfv)):=\langle\bfxi,\bfv\rangle \ .
\end{align*}
This is well defined. If $\varepsilon(\bfv)=0$, applying the inequality for
$C$ to $\bfv$ and $-\bfv$ gives $\langle\bfxi,\bfv\rangle=0$. Moreover, $\ell$ is
dominated by the continuous sublinear functional
\begin{align*}
     H(E):=\int_{Q_T} D(\hat P,E)\dd x\dd t,
 \qquad 0\leq H(E)\leq c_\K\Vert\hat P\Vert_{L^\infty(Q_T)}\Vert E\Vert_{L^1} \ .
\end{align*}
The Hahn--Banach theorem in its dominated extension form
\cite[Theorem 1.1]{Br11} extends $\ell$ to all of
$L^1(Q_T;\mathbb S)$, still with $\ell(E)\leq H(E)$. Applying this bound to
$E$ and $-E$ shows that the extension is continuous. Hence there is
$\bftau\in L^\infty(Q_T;\mathbb S)$ representing it. Testing the estimate
with constant symmetric matrices multiplied by characteristic functions,
first for matrices with rational entries and then by continuity for all
matrices, gives
\begin{align*}
 \bftau(x,t):E\leq D(\hat P(x,t),E)\quad\text{for all }E\in\mathbb S
 \quad\text{a.e.},
\end{align*}
and therefore $\bftau(x,t)\in\K(\hat P(x,t))$ a.e.
A related use of Hahn--Banach to construct a space-time flux is given in
\cite[Section 5.4]{KS22} for the total variation flow.
The identity
\begin{align*}
 \int_{Q_T}\bfxi\cdot \bfv\dd x\dd t
   =\int_{Q_T}\bftau:\varepsilon(\bfv)\dd x\dd t\quad\text{for all }\bfv\in C^1([0,T];X)
\end{align*}
first gives $-\div\bftau=\bfxi$ in the distributional sense. Applying it to
$\bfv(x,t)=\eta(t)\phi(x)$, and then using a countable dense subset of $X$,
gives the spatial Green identity for all $\phi\in X$ for almost every
$t$. This is the weak zero-traction condition on $\Gamma_N$. Thus
$\bftau\in\Sigma_{Q_T}[\hat P]$, proving \eqref{eq:wu-polar}.

By \eqref{eq:wu-polar} and the definition of $\D_{Q_T}[\hat P]$,
\begin{align*}
 F^{**}(\bfw)=\sup_{\bfxi\in C}\langle\bfxi,\bfw\rangle=\D_{Q_T}[\hat P](\bfw) \ .
\end{align*}
The Fenchel--Moreau theorem \cite[Chapter I]{ET99} identifies $F^{**}$ with the lower
semicontinuous relaxation of the proper convex functional $F$.
The norm closure of its epigraph therefore yields $\bfw_j\in C^1([0,T];X)$ with
$\bfw_j\rightarrow \bfw$ in $L^3(Q_T)^2$ and $F(\bfw_j)\rightarrow\D_{Q_T}[\hat P](\bfw)$. By
\eqref{eq:wu-F-smooth} we conclude $F(\bfw_j)=\D_{Q_T}[\hat P](\bfw_j)$, as claimed.
\end{proof}

\begin{lemma}[Parabolic recovery]\label{lem:wu-parabolic}
For every $\bfw\in L^3(Q_T)^2$ with $\D_{Q_T}[\hat P](\bfw)<\infty$, there is
a sequence $(\bfz_n)_n\subset C^1([0,T];X)$ such that
\begin{align}
 &\bfz_n\rightarrow \bfw\quad\text{in }L^3(Q_T)^2,\qquad
 \D_{Q_T}[\hat P](\bfz_n)\rightarrow\D_{Q_T}[\hat P](\bfw),\label{eq:wu-recovery-energy}\\
 &\Vert\bfz_n(0)-\bfu_0\Vert_{a(0)}\rightarrow0,\notag\\
 &\limsup_{n\rightarrow\infty}\int_{Q_T} a\,\partial_t \bfz_n\cdot(\bfz_n-\bfw)
       \dd x\dd t\leq0 \ .\label{eq:wu-recovery-time}
\end{align}
In particular, $\limsup_n\Hh_{\bfw}(\bfz_n)\leq0$.
\end{lemma}

\begin{proof}
Choose $\boldsymbol{b}_n\in C_c^\infty(\Omega)^2\subset X$ with
\begin{align*}
 \Vert\boldsymbol{b}_n-\bfu_0\Vert_{a(0)}\rightarrow0 \ .
\end{align*}
Such a sequence is obtained by truncation in the weighted $L^2$ norm and
smooth approximation of the bounded truncations in $L^2(\Omega)^2$, since
$a(0)\in L^\infty(\Omega)$.
In particular, $\D(\boldsymbol{b}_n)<\infty$.
Choose $0<h_n\leq 1/n$ small enough that
\begin{align}\label{eq:wu-scale}
 h_n\bigl(\Vert\boldsymbol{b}_n\Vert_{L^3}^3+\D(\boldsymbol{b}_n)\bigr)\leq1/n \ .
\end{align}
Lemma \ref{lem:wu-relaxation} allows
us to choose $\bfw_n\in C^1([0,T];X)$ such that
\begin{align}\label{eq:wu-input-recovery}
    \begin{split}
    \Vert\bfw_n-\bfw\Vert_{L^3(Q_T)}&\leq h_n^2,\\
    |\D_{Q_T}[\hat P](\bfw_n)-\D_{Q_T}[\hat P](\bfw)|&\leq1/n\ .
    \end{split}
\end{align}
Following the exponential time mollification in \cite[Section 3.2]{E25}, set
\begin{align}\label{eq:wu-filter}
 \bfz_n(t)=\exp(-t/h_n)\boldsymbol{b}_n
       +\frac{1}{h_n}\int_0^t \exp(-(t-s)/h_n)\bfw_n(s)\dd s \ .
\end{align}
This is an $X$-valued integral. We have $\bfz_n\in C^1([0,T];X)$, $\bfz_n(0)=\boldsymbol{b}_n$ and
\begin{align}\label{eq:wu-filter-derivative}
 \partial_t \bfz_n=(\bfw_n-\bfz_n)/h_n \ .
\end{align}
The convolution term in \eqref{eq:wu-filter} defines an operator on
$L^3(Q_T)^2$ with norm at most $1-\exp(-T/h_n)<1$ by Young's convolution
inequality \cite[Chapter 4]{Br11}.
After extension by zero to negative times, continuity of translations and
concentration of the kernel at zero show strong convergence to the identity
as $h_n\rightarrow0$. Direct integration and \eqref{eq:wu-scale} give
\begin{align*}
 \Vert \exp(-t/h_n)\boldsymbol{b}_n\Vert_{L^3(Q_T)}^3\leq\frac{h_n}{3}\Vert\boldsymbol{b}_n\Vert_{L^3}^3\leq\frac{1}{3n}\rightarrow0 \ .
\end{align*}
Together with \eqref{eq:wu-input-recovery}, this proves $\bfz_n\rightarrow \bfw$ in $L^3(Q_T)^2$.

We next control the dissipation by convexity and an exchange of the time
integrals, as in \cite[Lemma 3.3]{E25}. Convexity of $\D$ and the fact that
\begin{align*}
    \exp(-t/h_n)
       +\frac{1}{h_n}\int_0^t \exp(-(t-s)/h_n)\dd s  = 1
\end{align*}
imply by Jensen's inequality that
\begin{align*}
 \D(\bfz_n(t))\leq \exp(-t/h_n)\D(\boldsymbol{b}_n)
       +\frac{1}{h_n}\int_0^t \exp(-(t-s)/h_n)\D(\bfw_n(s))\dd s \ .
\end{align*}
Since $P$ is independent of time, the same spatial functional $\D$
occurs at $t$ and $s$. Multiplying by $\rho(t)$, using
$\rho(t)\leq\rho(s)$ for $s\leq t$, and applying Fubini gives
\begin{align}
 \D_{Q_T}[\hat P](\bfz_n)
 &\leq h_n\D(\boldsymbol{b}_n)\notag\\
 &\quad+\int_0^T\D(\bfw_n(s))
       \left(\int_s^T\frac{\rho(t)}{h_n}\exp(-(t-s)/h_n)\dd t\right)\dd s
       \notag\\
 &\leq h_n\D(\boldsymbol{b}_n)+\D_{Q_T}[\hat P](\bfw_n) \ .\label{eq:wu-filter-energy}
\end{align}
Thus $\limsup_n\D_{Q_T}[\hat P](\bfz_n)\leq\D_{Q_T}[\hat P](\bfw)$. Lower semicontinuity and $\bfz_n\rightarrow \bfw$
give the reverse bound, proving \eqref{eq:wu-recovery-energy}.

Finally, \eqref{eq:wu-filter-derivative} yields the exact identity
\begin{align}\label{eq:wu-time-defect}
 \int_{Q_T} a\,\partial_t \bfz_n\cdot(\bfz_n-\bfw)\dd x\dd t
 ={}&-h_n\int_{Q_T} a|\partial_t \bfz_n|^2\dd x\dd t\notag\\
   &+\int_{Q_T} a\,\partial_t \bfz_n\cdot(\bfw_n-\bfw)\dd x\dd t\ .
\end{align}
Both $(\bfw_n)$ and $(\bfz_n)$ are bounded in $L^3(Q_T)^2$, so
$\Vert\partial_t \bfz_n\Vert_{L^3(Q_T)}\leq C/h_n$. Since $|Q_T|<\infty$, H\"older's
inequality and \eqref{eq:wu-input-recovery} give
\begin{align*}
 &\left|\int_{Q_T} a\,\partial_t \bfz_n\cdot(\bfw_n-\bfw)\dd x\dd t\right|\\
 &\quad\leq\Vert a\Vert_{L^\infty(Q_T)} |Q_T|^{1/3}
       \Vert\partial_t \bfz_n\Vert_{L^3(Q_T)}\Vert\bfw_n-\bfw\Vert_{L^3(Q_T)}\\
 &\quad\leq C h_n\rightarrow0 \ .
\end{align*}
The first term on the right of \eqref{eq:wu-time-defect} is non-positive.
This proves \eqref{eq:wu-recovery-time}, including its initial-value claim.
\end{proof}

\subsubsection{Recovering the operator at the solution.}

\begin{lemma}\label{lem:wu-unshift}
Every weak solution in the sense of Definition \ref{def:solution_evolution}
satisfies, for all $\bfv\in C^1([0,T];X)$,
\begin{align}\label{eq:wu-unshifted}
    \begin{split}
    \Hh_{\bfu}(\bfv)&+\D_{Q_T}[\hat P](\bfv)-\D_{Q_T}[\hat P](\bfu)\\
    &+\int_0^T\rho(t)
       \langle\B(t;\bfu)-\bff,\bfv-\bfu\rangle\dd t\geq0\ .
    \end{split}
\end{align}
\end{lemma}

\begin{proof}
We combine the recovery sequence with a Minty perturbation argument
\cite[Section 2.3]{R13}.
For fixed $\bfu$, the functional $\bfv\mapsto\Hh_{\bfu}(\bfv)$ is convex on $C^1([0,T];X)$.
Integration by parts involving only the smooth test function gives
\begin{align}
 \Hh_{\bfu}(\bfv)
 ={}&\frac{1}{2}\Vert\bfv(T)\Vert_{a(T)}^2
     +\frac{1}{2}\int_{Q_T}(-\partial_ta)|\bfv|^2\dd x\dd t
     -\int_{Q_T} a\,\partial_t \bfv\cdot \bfu\dd x\dd t\notag\\
    &-\int_\Omega a(0)\bfv(0)\cdot \bfu_0\dd x
     +\frac{1}{2}\Vert\bfu_0\Vert_{a(0)}^2 \ .\label{eq:wu-H-convex}
\end{align}
All terms are finite. The quadratic terms are non-negative by
\eqref{eq:wu-a-monotone}, and the remaining terms are affine in $\bfv$,
which proves convexity.
Fix $\bfv\in C^1([0,T];X)$ and $0<\theta<1$. Apply Lemma \ref{lem:wu-parabolic} to $\bfw=\bfu$
and use the admissible test functions
\begin{align*}
 \bfv_{n,\theta}=(1-\theta)\bfz_n+\theta \bfv \ .
\end{align*}
Convexity of $\Hh_{\bfu}$ and $\D_{Q_T}[\hat P]$, followed by \eqref{eq:weak_solution}, gives
\begin{align*}
 0\leq{}&(1-\theta)
       \bigl[\Hh_{\bfu}(\bfz_n)+\D_{Q_T}[\hat P](\bfz_n)-\D_{Q_T}[\hat P](\bfu)\bigr]\\
 &+\theta\bigl[\Hh_{\bfu}(\bfv)+\D_{Q_T}[\hat P](\bfv)-\D_{Q_T}[\hat P](\bfu)\bigr]\\
 &+\int_0^T\rho(t)
       \langle\B(t;\bfv_{n,\theta})-\bff,\bfv_{n,\theta}-\bfu\rangle\dd t \ .
\end{align*}
Now $\bfv_{n,\theta}\rightarrow \bfu+\theta(\bfv-\bfu)$ in $L^3(Q_T)^2$. The continuity
$\B:L^3(Q_T)^2\rightarrow L^{3/2}(Q_T)^2$ and the recovery properties imply, on taking the upper
limit as $n\rightarrow\infty$,
\begin{align*}
 0\leq{}&\theta\bigl[\Hh_{\bfu}(\bfv)+\D_{Q_T}[\hat P](\bfv)-\D_{Q_T}[\hat P](\bfu)\bigr]\\
 &+\theta\int_0^T\rho(t)
       \langle\B(t;\bfu+\theta(\bfv-\bfu))-\bff,\bfv-\bfu\rangle\dd t \ .
\end{align*}
Divide by $\theta$ and let $\theta\downarrow0$, again using continuity
of $\B$. This proves \eqref{eq:wu-unshifted}.
\end{proof}

\subsubsection{Comparison of two weak solutions.}

\begin{proof}[Proof of Theorem \ref{th:rothe}(b)]
Let $\bfu_1,\bfu_2$ be two weak solutions for the same data. Put
\begin{align*}
 \bfw=\frac{\bfu_1+\bfu_2}{2},\qquad\delta=\bfu_1-\bfu_2 \ .
\end{align*}
Convexity gives $\D_{Q_T}[\hat P](\bfw)\leq\frac{1}{2}(\D_{Q_T}[\hat P](\bfu_1)+\D_{Q_T}[\hat P](\bfu_2))<\infty$.
Apply Lemma \ref{lem:wu-parabolic} to $\bfw$, with the common initial datum
$\bfu_0$, and denote the resulting sequence by $\bfz_n$.

Use $\bfv=\bfz_n$ in \eqref{eq:wu-unshifted} for both $\bfu_1$ and $\bfu_2$ and add.
The time and initial terms add to $2\Hh_{\bfw}(\bfz_n)$, hence
\begin{align*}
 0\leq{}&2\Hh_{\bfw}(\bfz_n)+2\D_{Q_T}[\hat P](\bfz_n)\\
 &-\D_{Q_T}[\hat P](\bfu_1)-\D_{Q_T}[\hat P](\bfu_2)\\
 &+\int_0^T\rho(t)\Big[
       \langle\B(t;\bfu_1)-\bff,\bfz_n-\bfu_1\rangle
       +\langle\B(t;\bfu_2)-\bff,\bfz_n-\bfu_2\rangle\Big]\dd t \ .
\end{align*}
By strong convergence $\bfz_n\rightarrow \bfw$ in $L^3(Q_T)^2$, the last line converges to
\begin{align*}
 -\frac{1}{2}\int_0^T\rho(t)
       \langle\B(t;\bfu_1)-\B(t;\bfu_2),\delta\rangle\dd t \ .
\end{align*}
The recovery properties and convexity of $\D_{Q_T}[\hat P]$ therefore give
\begin{align*}
 0
 &\leq 2\D_{Q_T}[\hat P](\bfw)-\D_{Q_T}[\hat P](\bfu_1)-\D_{Q_T}[\hat P](\bfu_2)\\
 &\qquad-\frac{1}{2}\int_0^T\rho(t)
           \langle\B(t;\bfu_1)-\B(t;\bfu_2),\delta\rangle\dd t\\
 &\leq-\frac{c_o}{4}\int_0^T\rho(t)\Vert\delta(t)\Vert_{L^3}^3\dd t
 \leq0 \ .
\end{align*}
Since $c_o>0$ and $\rho(t)\geq \exp(-\Vert\lambda\Vert_{L^1(0,T)})>0$, it follows that
$\delta=0$ in $L^3(Q_T)^2$. This proves uniqueness on all of $Q_T$. Existence follows from part (a).
\end{proof}

\begin{remark}
Time independence of $P$ is used in the dissipation estimate
\eqref{eq:wu-filter-energy}. No positive lower bound for $P$ or $m$, no
Korn or Poincar\'e inequality, and no time derivative of either solution
are needed. The approximation order is essential: first choose $\boldsymbol{b}_n$,
then $h_n$, and only then a sufficiently accurate recovery $\bfw_n$.
This avoids requiring a convergence rate in Lemma \ref{lem:wu-relaxation}.
The assertion concerns uniqueness of the velocity only.
\end{remark}

The combination of time mollification and energy approximation has related
uses in the scalar $BV$ setting of \cite[Sections 3 and 4]{E25}. The lemmas above establish
the approximation and comparison steps directly for the weighted
vector-valued formulation used here.

\section{Open problems}
The main open problem remains the coupling
$\bfu\mapsto(h,A)\mapsto(P,m)$ through the transport equations
\begin{align*}
    \partial_t h+\div(h\bfu)&=S_h,\\
    \partial_t A+\div(A\bfu)&=S_A \ .
\end{align*}
The velocity is a priori only in $L^3(Q_T)$. Under the additional spatial
Lipschitz assumption on $P$, Theorem \ref{th:measure_structure} gives local
$BD$ regularity on $\{P>0\}$ for spatial fields of finite dissipation and
a bound on the negative part of their weighted divergence. The evolution
theorem, however, permits merely continuous time-dependent $P$. A
corresponding space-time measure statement under those general assumptions
is not asserted here. Even with additional regularity of $P$, its use for
the evolution problem requires a space-time argument or a suitable
identification and estimate of the time-slice dissipations. Such information
does not by itself supply the velocity regularity needed for a coupling
theorem. The strong well-posedness results in \cite{LTT22,BDHH22,B25} concern
regularised or modified stresses in higher-order Sobolev settings.

Theorem \ref{th:rothe} establishes existence of weak variational solutions
for time-dependent coefficients and uniqueness for time-independent $P$.
Weak uniqueness for general time-dependent $P$ remains open in this
framework. In the proof of part (b),
time independence is used to control the dissipation of a one-sided time
convolution. This estimate is not available from continuity of $P$ alone.

Lemma \ref{lem:stress_reconstruction} identifies an admissible saturated
stress for the time-discrete variational problem. A full stress-and-flow-rule
formulation for the weak evolution solutions, including the interpretation
of any internal interface conditions, requires further analysis. Related
duality and relaxation methods in plasticity are developed in \cite{RG80,KT83}.

\section{Appendix}

\subsection{Normal trace}
The construction follows the slicing technique for normal traces of Anzellotti \cite{A83}. See also the trace theory of the Hencky stress spaces in \cite{KT83}.

The assumption (A1) is central for this result. It implies $\Gamma_N\cap\bar\Gamma_D = \emptyset$,
so $\bar\Gamma_D \setminus \Gamma_D \subseteq \partial\Omega\setminus(\Gamma_D\cup\Gamma_N)$ and therefore $\mathcal{H}^1(\bar\Gamma_D\setminus\Gamma_D) = 0$.
\begin{definition}
    For $\bftau\in L^\infty(\Omega;\mathbb{S})$ with $\div\bftau\in V^*$, define
    \begin{align}
        \langle \bftau\nu, \bfw \rangle := \int_\Omega \bftau : \varepsilon(\bfw) \dd x + \int_\Omega \bfw \cdot \div \bftau \dd x, \quad\text{for all }\bfw\in W^{1,3}(\Omega)^2\ . \label{eq:normaltrace}
    \end{align}
\end{definition}
Both integrals are finite, and for smooth $\bftau$ and $\bfw$ we have
$\langle \bftau\nu, \bfw \rangle = \int_{\partial\Omega} \bfw \cdot (\bftau \nu) \dd \mathcal{H}^1$ by the Gauss theorem. Thus \eqref{eq:normaltrace} is a distributional normal trace. For $\bftau \in \Sigma$ we have
$\langle\bftau\nu,\bfw\rangle = 0$ for all $\bfw \in C^1(\bar\Omega; \mathbb R^2)$ with $\bfw_{|\Gamma_D} =0$.

\begin{lemma}\label{lem:a2}
    Let $\bftau \in \Sigma$ and $\bfw \in W^{1,3}(\Omega)^2$ with $\bfw =0$ a.e.\ in $U\cap \Omega$ for an open set $U\subseteq \mathbb R^2$
    with $\bar\Gamma_D \subset U$. Then $\langle \bftau\nu, \bfw \rangle =0$.
\end{lemma}
\begin{proof}
    Since $\bar\Gamma_D$ is compact and $U$ open, there exists an $\eta'>0$ with $U_{3\eta'}(\bar\Gamma_D) \subset U$.
    Let $E:W^{1,3}(\Omega) \rightarrow W^{1,3}(\mathbb R^2)$ be an extension operator and
    $\chi\in\mathrm{Lip}(\mathbb R^2)$ with $\chi=0$ in $U_{\eta'}(\bar\Gamma_D)$
    and $\chi=1$ in $\mathbb R^2 \setminus U_{2\eta'}(\bar\Gamma_D)$, $0\leq \chi\leq 1$. Set $g := \chi E\bfw \in W^{1,3}(\mathbb R^2)$. Then $g=\bfw$
    in $\Omega$, and for $\rho<\eta'/2$, $g_\rho:= g * \varphi_\rho\in C^\infty(\mathbb{R}^2)$ with $g_\rho =0$ in $U_{\eta'/2}(\bar\Gamma_D)$ and $\varphi_\rho$ the standard mollifier. In particular $g_{\rho|\Gamma_D} = 0$
    and $g_{\rho|\bar\Omega} \in C^1(\bar\Omega)$. Now it follows from the definition
    that $\langle\bftau\nu, g_\rho\rangle =0$, and since $g_\rho \rightarrow g$ in
    $W^{1,3}(\mathbb R^2)$, also $g_{\rho|\Omega} \rightarrow \bfw $ in $W^{1,3}(\Omega)$. Thus both integrals converge in \eqref{eq:normaltrace} and
    $\langle\bftau\nu,\bfw\rangle =0$.
\end{proof}
\begin{lemma}\label{lem:a3}
    Under Assumption \ref{ass}, for every $\bftau\in\Sigma$ and $\bfv \in W^{1,3}(\Omega)^2$
    \begin{align*}
        - \langle\bftau\nu,\bfv\rangle \leq \int_{\Gamma_D} D(P, -(\bfv\odot \nu)) \dd \mathcal{H}^1
    \end{align*}
    where $\bfv$ on $\Gamma_D$ is evaluated in the sense of traces.
\end{lemma}
\begin{proof}
    Step 1:
    Let $\eta>0$ and $\zeta_\eta(x) := \psi (\dist(x,\bar\Gamma_D)/\eta)$ with $\psi \in \mathrm{Lip}(\mathbb R)$
    and $\psi = 1$ in $(-\infty,1/2]$, $\psi=0$ in $[1,\infty)$ and $0\leq \psi \leq 1$.
    Then $(1-\zeta_\eta)\bfv =0$ in $U_{\eta/2}(\bar\Gamma_D)\cap \Omega$. By Lemma \ref{lem:a2} we have
    $\langle \bftau \nu, \bfv\rangle = \langle \bftau\nu,\zeta_\eta \bfv\rangle$. We will prove that
    \begin{align}
        -\langle\bftau\nu,\bfw\rangle \leq \int_{\partial\Omega} D(P,-(\bfw\odot \nu)) \dd \mathcal{H}^1 \label{eq:step1}
    \end{align}
    for $\bfw:=\zeta_\eta\bfv \in W^{1,3}(\Omega)^2$, and we take the limit as $\eta\rightarrow0$
    in Step 6.

    Step 2:
    Assumption \ref{ass} (A1) implies that there exist finitely many open cylinders $Z_1,\dots,Z_N \subset \mathbb R^2$
    which cover $\partial\Omega$, such that (after a rigid motion)
    \begin{align*}
        \Omega\cap Z_j = \{y=(y_1,y_2)\in Z_j : y_2>\gamma_j(y_1)\}, \quad \partial\Omega\cap Z_j
        =\{y:y_2=\gamma_j(y_1)\}\cap Z_j
    \end{align*}
    with $L$-Lipschitz continuous functions $\gamma_j$. The outer normal and the surface element
    in these coordinates are, for $\mathcal{L}^1$ a.e.\ $y_1$,
    \begin{align*}
        \nu(y_1) = \frac{(\gamma_j'(y_1),-1)}{\sqrt{1+\gamma'_j(y_1)^2}}, \quad \dd \mathcal{H}^1 = \sqrt{1+\gamma'_j(y_1)^2} \dd y_1 \ .
    \end{align*}
    Both are invariant under vertical translations. Choose $Z_0\Subset \Omega$ with
    $\bar\Omega \subset \cup_{j=0}^N Z_j$ and a subordinate smooth partition of unity $(\theta_j)_{j=0}^N$,
    $\theta_j\in C^\infty_c(Z_j)$ and $\sum_{j=0}^{N}\theta_j =1$ on $\bar\Omega$.

    Since \eqref{eq:normaltrace} is linear in $\bfw$ and
    $\theta_j\bfw \in W^{1,3}(\Omega)$, we notice that \eqref{eq:step1} can be shown term by term.
    The term for $j=0$ vanishes, since $\theta_0\bfw\in W^{1,3}$ has compact support
    in $\Omega$. As in the proof of Lemma \ref{lem:a2}, the mollified functions have zero trace on $\partial\Omega$ and the limit leads to
    $\langle \bftau\nu, \theta_0\bfw\rangle =0$.

    From now on let $j\geq 1$ be fixed, $\gamma:= \gamma_j$, and redefine $\bfw:= \theta_j\zeta_\eta\bfv$. Choose an open interval $B' \Subset \{y_1:(y_1,y_2)\in Z_j\ \text{for a}\ y_2\}$, a height $M$ and $t_0>0$ with
    \begin{align*}
        \supp \bfw \cap \bar\Omega \subset \{y : y_1\in B', \gamma(y_1)\leq y_2 < M-t_0\} \Subset Z_j \ .
    \end{align*}
    Step 3: For $t \in (0,t_0)$ we define
    \begin{align*}
        D_t := \{y : y_1\in B', \gamma(y_1)+t <y_2<M\}, \quad S_t =\{y : y_1\in B', y_2 = \gamma(y_1)+t\} \ .
    \end{align*}
    $D_t$ is a Lipschitz domain and $S_t$ its lower boundary. We notice that $\bfw$ vanishes for $y_1\in\partial B'$ and $y_2=M$. For $y\in D_t$ a simple calculation shows that $\dist(y,\partial\Omega)\geq c_L t$ for $c_L := (1+L^2)^{-1/2}$ and
    $L$ the Lipschitz constant of $\gamma$.

    Let $\bftau_\rho := \bftau * \varphi_\rho$ with $\varphi_\rho$ the standard mollifier and $\rho < c_L t/2$. On a neighbourhood of $\bar D_t$, $\bftau_\rho$ is
    smooth and we have
     \begin{align*}
        \div \bftau_\rho = (\div\bftau) * \varphi_\rho \ ,
     \end{align*}
     and with the outer normal $\nu(y_1)$ of $D_t$ on $S_t$ (identical to the one of $\Omega$, by
     translation invariance)
     \begin{align}
        \int_{D_t} \bftau_\rho : \varepsilon(\bfw) + \bfw \cdot \div\bftau_\rho \dd y =
        \int_{S_t} \bfw \cdot(\bftau_\rho \nu) \dd \mathcal{H}^1 \ . \label{eq:trans}
     \end{align}
     Step 4:
     The function $T(y_1, s):= (y_1, \gamma(y_1)+s)$ is bi-Lipschitz from $B'\times (0,t_0)$ to
     \begin{align*}
        C := T(B'\times (0,t_0)) = \{y : y_1\in B', 0<y_2-\gamma(y_1)<t_0\} = \cup_{0<t<t_0} S_t
     \end{align*}
     and measure preserving ($\det \mathrm{D}T=1$). In the following we will apply
     a Fubini argument to select admissible slices.
     Let $N_1 \subset \Omega$ be the $\mathcal{L}^2$ null set of non-Lebesgue points of $\bftau$,
     $N_2$ the null set where $\bftau(x)\in \K(P(x))$ fails and $N_3$ the null set where $\bfw \circ T$ is not absolutely continuous on lines (ACL) in $s$-direction. With Fubini we have for $\mathcal{L}^1$ a.e.\ $t\in(0,t_0)$ that $\mathcal{H}^1$ a.e.\ points of $S_t$ are Lebesgue points of $\bftau$ and admissible.

    For such a $t$ we take the limit as $\rho \rightarrow 0$ in \eqref{eq:trans}.
    For the right-hand side, $\bftau_\rho \rightarrow \bftau$ pointwise $\mathcal{H}^1$ a.e.\ on $S_t$ with majorant $\Vert \bftau\Vert_{L^\infty}|\bfw_{|S_t}| \in L^1(S_t)$.
    With the pointwise bound of the support function (Lemma \ref{lem:prop_D}) we have
    \begin{align*}
        \bfw \cdot (\bftau\nu) = \bftau : (\bfw\odot \nu) \geq - D(P, -(\bfw\odot \nu))
    \end{align*}
    $\mathcal{H}^1$ a.e.\ on $S_t$. For a.e.\ $t\in(0,t_0)$ we therefore have
    \begin{align}
        -\int_{D_t} \bftau : \varepsilon(\bfw) + \bfw\cdot\div\bftau \dd y
        \leq \int_{S_t} D(P, -(\bfw\odot \nu)) \dd \mathcal{H}^1 \ . \label{eq:slice}
    \end{align}
    Step 5:
    For the left-hand side, the monotone exhaustion $D_t \nearrow D_0$ with
    $L^1$ integrand leads to convergence to $-\langle \bftau\nu, \bfw\rangle$, since
    the support of $\bfw$ makes the $D_0$ integral equal to the $\Omega$ integral.
    For the right-hand side we use the translation invariance of $\nu$ and $\dd \mathcal{H}^1$:
    \begin{align*}
        &\int_{S_t} D(P, -(\bfw\odot \nu)) \dd \mathcal{H}^1\\
        &\quad=\int_{B'} D(P(T(y_1,t)), -(\bfw (T(y_1,t))\odot \nu(y_1))) \sqrt{1+\gamma'(y_1)^2} \dd y_1 \ .
    \end{align*}
    By ACL we have for a.e.\ $y_1$
    \begin{align*}
        |\bfw (T(y_1,t)) - \bfw (T(y_1,0^+))| \leq \int_{0}^{t} |\partial_s(\bfw\circ T)(y_1,s)| \dd s,
    \end{align*}
    and integration over $y_1$ and H\"older lead to
    \begin{align*}
        \Vert \bfw \circ T(\cdot, t) - \bfw \circ T(\cdot, 0^+)\Vert_{L^1(B')}
        \leq  C t^{2/3} \Vert \nabla(\bfw \circ T)\Vert_{L^3} \rightarrow 0 \ .
    \end{align*}
    The ACL boundary values $\bfw \circ T(\cdot, 0^{+})$ coincide $\mathcal{L}^{1}$ a.e.\ with the trace $\bfw_{|\partial\Omega}$.
    Also $|P(T(y_1,t))-P(T(y_1,0))| \leq \omega_P(t)$, where $\omega_P$ is the modulus of continuity of $P$.
    The Lipschitz continuity from Lemma \ref{lema:prop_K} and Lemma \ref{lem:prop_D} lead to
    \begin{align*}
        |D(P_1,\bfxi_1) - D(P_2,\bfxi_2)| \leq c_\K |P_1-P_2| |\bfxi_1|_F + c_\K \Vert P\Vert_{L^\infty}
        |\bfxi_1-\bfxi_2|_F \ .
    \end{align*}
    This and $|a\odot b|_F \leq |a||b|$ lead to the convergence of the right-hand side of \eqref{eq:slice}
    to $\int_{\partial\Omega} D(P,-(\bfw\odot \nu)) \dd \mathcal{H}^1$ along a sequence of admissible $t$. This proves \eqref{eq:step1}, after summation over $j$.

    Step 6:
    Returning to Step 1 and \eqref{eq:step1}, we have with $\bfw = \zeta_\eta\bfv$ and the positive homogeneity
    of degree 1 of $D(P,\cdot)$ (Lemma \ref{lem:prop_D}) and $0\leq\zeta_\eta\in\mathrm{Lip}(\bar\Omega)$
    \begin{align*}
        -\langle \bftau\nu, \bfv\rangle = - \langle \bftau\nu, \zeta_\eta \bfv\rangle
        \leq \int_{\partial\Omega} \zeta_\eta D(P, -(\bfv\odot \nu)) \dd \mathcal{H}^1 \ .
    \end{align*}
    As $\eta\rightarrow 0$, $\zeta_\eta \rightarrow \mathbf 1_{\bar\Gamma_D}$ pointwise on $\partial\Omega$.
    The integrands are bounded by the $\mathcal{H}^{1}$-integrable function
    $c_\K \Vert P\Vert_{L^{\infty}}|\bfv_{|\partial\Omega}|$.
    Dominated convergence and $\mathcal{H}^{1}(\bar\Gamma_D\setminus\Gamma_D) =0$ conclude the proof.
\end{proof}
\begin{lemma} \label{lem:prove_leq}
    Under Assumption \ref{ass}, for every $\bfv \in W^{1,3}(\Omega)^2$
    \begin{align*}
        \D(\bfv) \leq \int_\Omega D(P,\varepsilon(\bfv)) \dd x + \int_{\Gamma_D} D(P, -(\bfv\odot \nu) ) \dd \mathcal{H}^1 \ .
    \end{align*}
\end{lemma}
\begin{proof}
    For $\bftau \in \Sigma$, Lemma \ref{lem:prop_D} gives a pointwise estimate for the volume term,
    and together with Lemma \ref{lem:a3} this leads to
    \begin{align*}
        - \int_\Omega \bfv \cdot \div \bftau \dd x &= \int_\Omega \bftau : \varepsilon(\bfv) \dd x
        - \langle \bftau\nu, \bfv \rangle\\
        &\leq \int_\Omega D(P, \varepsilon(\bfv)) \dd x
        +\int_{\Gamma_D} D(P, -(\bfv\odot \nu) ) \dd \mathcal{H}^{1} \ .
    \end{align*}
    Taking the supremum over $\bftau\in \Sigma$ concludes the proof.
\end{proof}
\bibliographystyle{plain} 
\bibliography{lit}

@article{RG80,
  title={Duality and relaxation in the variational problems of plasticity},
  author={Temam, Roger and Strang, Gilbert},
  journal={Journal de M{\'e}canique},
  volume={19},
  number={3},
  pages={493--527},
  year={1980}
}

@article {H79,
      author = "W. D.  Hibler ",
      title = "A Dynamic Thermodynamic Sea Ice Model",
      journal = "Journal of Physical Oceanography",
      year = "1979",
      publisher = "American Meteorological Society",
      address = "Boston MA, USA",
      volume = "9",
      number = "4",
      doi = "10.1175/1520-0485(1979)009<0815:ADTSIM>2.0.CO;2",
      pages=      "815--846",
      url = "https://journals.ametsoc.org/view/journals/phoc/9/4/1520-0485_1979_009_0815_adtsim_2_0_co_2.xml"
}

@book {Z85,

    AUTHOR = {Zeidler, Eberhard},
     TITLE = {Nonlinear functional analysis and its applications. {III}},
      NOTE = {Variational methods and optimization,
              Translated from the German by Leo F. Boron},
 PUBLISHER = {Springer-Verlag, New York},
      YEAR = {1985},
     PAGES = {xxii+662},
      ISBN = {0-387-90915-X},
   MRCLASS = {49-02 (46G99 47Hxx 58-02 90Cxx)},
  MRNUMBER = {768749},
MRREVIEWER = {Jean\ Mawhin},
       DOI = {10.1007/978-1-4612-5020-3},
       URL = {https://doi.org/10.1007/978-1-4612-5020-3},
}

@article{A83,
  title={Pairings between measures and bounded functions and compensated compactness},
  author={Anzellotti, Gabriele},
  journal={Annali di Matematica Pura ed Applicata (4)},
  volume={135},
  pages={293--318},
  year={1983},
  doi={10.1007/BF01781073}
}

@article{KT83,
  title={Dual spaces of stresses and strains, with applications to {H}encky plasticity},
  author={Kohn, Robert V. and Temam, Roger},
  journal={Applied Mathematics and Optimization},
  volume={10},
  number={1},
  pages={1--35},
  year={1983},
  doi={10.1007/BF01448377}
}

@article{M34,
  title={Extension of range of functions},
  author={McShane, Edward J.},
  journal={Bulletin of the American Mathematical Society},
  volume={40},
  number={12},
  pages={837--842},
  year={1934},
  doi={10.1090/S0002-9904-1934-05978-0}
}

@article{LTT22,
  title={Well-posedness of {H}ibler's dynamical sea-ice model},
  author={Liu, Xin and Thomas, Marita and Titi, Edriss S.},
  journal={Journal of Nonlinear Science},
  volume={32},
  number={4},
  pages={Paper No. 49},
  year={2022},
  doi={10.1007/s00332-022-09803-y}
}

@article{BDHH22,
  title={Rigorous analysis and dynamics of {H}ibler's sea ice model},
  author={Brandt, Felix and Disser, Karoline and Haller-Dintelmann, Robert and Hieber, Matthias},
  journal={Journal of Nonlinear Science},
  volume={32},
  number={4},
  pages={Paper No. 50},
  year={2022},
  doi={10.1007/s00332-022-09805-w}
}

@misc{DGH25,
  author = {Denk, Robert and Gmeineder, Franz and Hieber, Matthias},
  title = {On the Singular Limit in {H}ibler's Sea Ice Model},
  year = {2025},
  eprint = {2511.09327},
  archivePrefix = {arXiv},
  primaryClass = {math.AP},
  note = {arXiv:2511.09327},
  doi = {10.48550/arXiv.2511.09327}
}

@misc{DD25,
  author = {Dingel, Stefan and Disser, Karoline},
  title = {Global existence and uniqueness for {H}ibler's visco-plastic sea-ice model},
  year = {2025},
  eprint = {2508.16537},
  archivePrefix = {arXiv},
  primaryClass = {math.AP},
  note = {arXiv:2508.16537},
  doi = {10.48550/arXiv.2508.16537}
}

@article{B25,
  author = {Brandt, Felix},
  title = {Well-posedness of {H}ibler's parabolic-hyperbolic sea ice model},
  journal = {Journal of Evolution Equations},
  volume = {25},
  pages = {Paper No. 82},
  year = {2025},
  doi = {10.1007/s00028-025-01098-2}
}

@misc{E25,
  author = {Elenius, Theo},
  title = {Characterizations and properties of solutions to parabolic problems of linear growth},
  year = {2025},
  eprint = {2505.16747},
  archivePrefix = {arXiv},
  primaryClass = {math.AP},
  note = {arXiv:2505.16747v2},
  doi = {10.48550/arXiv.2505.16747}
}

@book{ET99,
  author = {Ekeland, Ivar and Temam, Roger},
  title = {Convex Analysis and Variational Problems},
  series = {Classics in Applied Mathematics},
  volume = {28},
  publisher = {Society for Industrial and Applied Mathematics},
  address = {Philadelphia},
  year = {1999},
  note = {Reprint of the 1976 English edition},
  doi = {10.1137/1.9781611971088}
}

@book{Br11,
  author = {Brezis, Haim},
  title = {Functional Analysis, {Sobolev} Spaces and Partial Differential Equations},
  series = {Universitext},
  publisher = {Springer},
  address = {New York},
  year = {2011},
  doi = {10.1007/978-0-387-70914-7}
}

@book{R13,
  author = {Roub{\'i}{\v c}ek, Tom{\'a}{\v s}},
  title = {Nonlinear Partial Differential Equations with Applications},
  series = {International Series of Numerical Mathematics},
  volume = {153},
  edition = {Second},
  publisher = {Birkh{\"a}user},
  address = {Basel},
  year = {2013},
  doi = {10.1007/978-3-0348-0513-1}
}

@article{KS22,
  author = {Kinnunen, Juha and Scheven, Christoph},
  title = {On the definition of solution to the total variation flow},
  journal = {Calculus of Variations and Partial Differential Equations},
  volume = {61},
  number = {1},
  pages = {Paper No. 40},
  year = {2022},
  url = {https://arxiv.org/abs/2106.05711}
}

@article{HD97,
  author = {Hunke, Elizabeth C. and Dukowicz, John K.},
  title = {An Elastic--Viscous--Plastic Model for Sea Ice Dynamics},
  journal = {Journal of Physical Oceanography},
  volume = {27},
  number = {9},
  pages = {1849--1867},
  year = {1997},
  doi = {10.1175/1520-0485(1997)027<1849:AEVPMF>2.0.CO;2}
}

@article{H01,
  author = {Hunke, Elizabeth C.},
  title = {Viscous--Plastic Sea Ice Dynamics with the {EVP} Model: Linearization Issues},
  journal = {Journal of Computational Physics},
  volume = {170},
  number = {1},
  pages = {18--38},
  year = {2001},
  doi = {10.1006/jcph.2001.6710}
}

@article{BFLM13,
  author = {Bouillon, Sylvain and Fichefet, Thierry and Legat, Vincent and Madec, Gurvan},
  title = {The elastic--viscous--plastic method revisited},
  journal = {Ocean Modelling},
  volume = {71},
  pages = {2--12},
  year = {2013},
  doi = {10.1016/j.ocemod.2013.05.013}
}

@article{BLTT26a,
  author = {Boutros, Daniel W. and Liu, Xin and Thomas, Marita and Titi, Edriss S.},
  title = {Global well-posedness of the elastic--viscous--plastic sea-ice model with the inviscid {Voigt}-regularisation},
  journal = {Mathematical Models and Methods in Applied Sciences},
  volume = {36},
  number = {9},
  pages = {1935--1965},
  year = {2026},
  doi = {10.1142/S0218202526500296},
  url = {https://arxiv.org/abs/2505.03080}
}

@misc{BLTT26b,
  author = {Boutros, Daniel W. and Liu, Xin and Thomas, Marita and Titi, Edriss S.},
  title = {A mathematical study of an elastic-viscous-plastic sea-ice model with the {Kelvin--Voigt} rheology},
  year = {2026},
  eprint = {2604.26295},
  archivePrefix = {arXiv},
  primaryClass = {math.AP},
  note = {arXiv:2604.26295},
  doi = {10.48550/arXiv.2604.26295}
}
\end{document}